\documentclass[a4paper]{amsart}
\usepackage{amssymb}
\usepackage{amscd}
\usepackage{amsthm}
\usepackage{amsmath}
\usepackage{mathtools}
\usepackage{latexsym}
\usepackage[all]{xy}
\usepackage[utf8]{inputenc}
\usepackage{enumitem}
\usepackage{xcolor}
\newif\ifshowchanges
\showchangestrue

\makeatletter
\@namedef{subjclassname@2020}{\textup{2020} Mathematics Subject Classification}
\makeatother

\setlist[enumerate]{label={(\roman*)}}%,parsep=\smallskipamount,itemsep=0pt}

\theoremstyle{plain}
\newtheorem{theorem}{Theorem}
\newtheorem{corollary}[theorem]{Corollary}
\newtheorem{lemma}[theorem]{Lemma}
\newtheorem{proposition}[theorem]{Proposition}

\theoremstyle{definition}
\newtheorem{definition}[theorem]{Definition}
\newtheorem{example}[theorem]{Example}

\theoremstyle{remark}

\newtheorem*{remark}{Remark}

\numberwithin{theorem}{section}

\newcommand{\Spec}[1]{\operatorname{Spec}(#1)}

\newcommand{\card}{\mbox{\rm{card\,}}}
\newcommand{\Add}{\mbox{\rm{Add\,}}}

\newcommand{\Ass}[2]{\mbox{\rm{Ass}}_{#1}(#2)}

\newcommand{\Hom}[3]{\operatorname{Hom}_{#1}(#2,#3)}
\newcommand{\Ext}[4]{\operatorname{Ext}^{#1}_{#2}(#3,#4)}

\newcommand{\rfmod}[1]{\mbox{\rm{mod}--}{#1}}
\newcommand{\rmod}[1]{\mbox{\rm{Mod}--}{#1}}
\newcommand{\lfmod}[1]{{#1}\mbox{--\rm{mod}}}
\newcommand{\lmod}[1]{{#1}\mbox{--\rm{Mod}}}

\newcommand{\ra}{\rightarrow}

\begin{document}

\title[Balanced pairs]{Balanced pairs, virtually Gorenstein rings, and cotorsion torsion triples}

\author{Sergio Estrada}
\address{Departamento de Matem$\acute{a}$ticas, Universidad de Murcia, 30100 Murcia, Spain}
\email{sestrada@um.es}

\author{Jiangsheng Hu}
\address{School of Mathematics, Hangzhou Normal University, Hangzhou 311121, China}
\email{hujs@hznu.edu.cn}

\author{Jan Trlifaj}
\address{Charles University, Faculty of Mathematics
and Physics, Department of Algebra, 186 75 Prague 8, Czech Republic}
\email{trlifaj@karlin.mff.cuni.cz}

\begin{abstract} For any ring $R$, we investigate balanced pairs of classes of modules and their relations to cotorsion
triples. We characterize the special balanced pairs that fit into complete hereditary cotorsion triples. As an application, we prove that Gorenstein projective and Gorenstein injective modules form a balanced pair if and only if $R$ is right virtually Gorenstein. We also characterize the tilting and cotilting cotorsion pairs arising from balanced pairs.

In \cite{BBOS}, cotorsion torsion triples in abelian categories were employed in the representation theory of rectangular grids occurring in persistent homology theory. For module categories, we use infinite dimensional tilting theory to completely classify all cotorsion torsion triples by means of $1$-resolving subcategories of $\rfmod R$, and to give an explicit monic map between the formally dual notions of cotorsion torsion triples of right $R$-modules and torsion cotorsion triples of left $R$-modules. This map is bijective in case the underlying ring $R$ is left noetherian, but not in general.
\end{abstract}

\date{\today}

\thanks{The first author was supported by grant 22004/PI/22 funded by Fundaci\'on S\'eneca, Agencia de Ciencia y Tecnolog\'ia de la Regi\'on de Murcia and by grant PID2024-155576NB-I00 funded by MICIU/AEI/10.13039/501100011033 /FEDER, UE. The second author was supported by the NSF of China (Grant No. 12571035) and Jiangsu 333 Project. Research of the third author supported by GA\v CR 26-22734S}
	

\maketitle

\section{Introduction}

This paper investigates the role of tilting and cotilting classes in various settings of relative homological algebra. The first setting concerns the balance for computing derived functors of the bifunctor $\Hom{R}{-}{-}$: a pair of classes of modules $\mathcal{F}\times\mathcal{G}$, where $\mathcal{F}$ is precovering and $\mathcal{G}$ is preenveloping, is said to be \emph{balanced} provided that the derived functors can equivalently be computed using $\mathcal F$-resolutions and $\mathcal G$-coresolutions, respectively.
 
We show that under the hypotheses of Proposition~\ref{balanced-hered}(1), balanced pairs always induce hereditary cotorsion triples, and by Proposition~\ref{balanced-structural}, all hereditary cotorsion triples are complete and have definable middle terms. This yields a Salce-type correspondence for balanced pairs: under these hypotheses, $\mathcal F$ is special precovering, if and only if $\mathcal G$ is special preenveloping (Theorem~\ref{hered}), cf.\ \cite[Salce Lemma 5.20]{GT}. We also consider the question of whether the associated classes $\mathcal{F}$ and $\mathcal{G}$ generate and cogenerate tilting and cotilting cotorsion pairs, and characterize such pairs in Corollaries~\ref{tilt} and \ref{cotilt}. 

\medskip
In the course of our work, we discover three interesting new applications of balanced pairs in relative homological algebra. 

Our main application concerns Gorenstein analogues of projective and injective modules, the Gorenstein projective and Gorenstein injective modules. Let $R$ be a ring and let $\rmod R$ (resp. $\lmod R$) denote the category of all right (resp. left) $R$-modules. It is well known that the pair of classes of projective and injective $R$-modules $\mathcal{P}_{0}\times\mathcal{I}_{0}$ forms a balanced pair in $\rmod R$; in other words, the functors $\textrm{Ext}^i_{R}(-,-)$ can be computed using either projective resolutions or injective coresolutions. 

However, in Gorenstein homological algebra, the characterization of the rings $R$ for which the pair $\mathcal{GP}\times\mathcal{GI}$, of Gorenstein projective and Gorenstein injective modules, is balanced in $\rmod R$ has remained open. Here we solve this question for any ring $R$: $\mathcal{GP}\times\mathcal{GI}$ is balanced if and only if $R$ is right \emph{virtually Gorenstein} (Corollary~\ref{cor-GP-virtually-Gor}).  Recall that virtually Gorenstein algebras were introduced by Reiten and Beligiannis in \cite{BR07} in the context of artin algebras and later extended to arbitrary rings in \cite{EPZ,DLW}. This result shows that virtually Gorenstein rings provide the appropriate framework for a well-behaved version of Gorenstein homological algebra. Indeed, for such rings, the class $\mathcal{GP}$ is known to be special precovering, and every Gorenstein projective module is Gorenstein flat (\cite{WE25})---properties that are classical questions in Gorenstein homological algebra. 

Our second application concerns Ding modules. Let $\mathcal{DP}$,
$\mathcal{DI}$, and $\mathcal{PGF}$ denote the classes of Ding
projective, Ding injective, and projectively coresolved Gorenstein
flat right $R$-modules, respectively; see
\cite{DLM09,DM08,Gil10,SaS}. We prove that
$\mathcal{DP}\times\mathcal{DI}$ is balanced if and only if
$(\mathcal{DP},\mathcal{DP}^{\perp_1},\mathcal{DI})$ is a complete
hereditary cotorsion triple (Corollary~\ref{ding-balance-arbitrary}).
In this case, $\mathcal{DP}=\mathcal{PGF}$, so every Ding projective
module is Gorenstein flat (Corollary~\ref{ding-balance-pgf}).
Furthermore, Corollaries~\ref{cor-GP-virtually-Gor} and
\ref{balanced-GP-DP-PGF-GI-DI} give
$\mathcal{GP}=\mathcal{DP}=\mathcal{PGF}$ and
$\mathcal{GI}=\mathcal{DI}$ for every right virtually Gorenstein
ring. Gillespie established the
equalities $\mathcal{GP}=\mathcal{DP}$ and
$\mathcal{GI}=\mathcal{DI}$ over Ding--Chen rings, c.f. \cite[Theorem~1.1]{GillespieDing}. These rings are right
virtually Gorenstein by \cite[Theorem 7]{DC96} and
\cite[Theorem A]{DLW}, and Example~\ref{vG}(2) shows that
the inclusion is strict.

Our third application shows that if $R$ is a noetherian non-perfect ring, then the natural covering class of Gorenstein flat modules can never be the left-hand part of a balanced pair (Proposition \ref{prop:GF-balanced-pair}). This result is the Gorenstein analogue of the lack of balance for the class of flat modules, initially proved by Enochs in \cite{Enochs15} for commutative noetherian rings and later extended in \cite{EPZ} to noncommutative noetherian rings.

In the second part of the paper, we return to our main motivation: the study of tilting and cotilting classes with respect to orthogonal pairs and triples. Recall that a \emph{torsion pair} $(\mathcal{T},\mathcal{F})$ in $\rmod R$ gives a splitting of the category into a torsion class $\mathcal{T}$ and a torsion-free class $\mathcal{F}$; a \emph{cotorsion pair} $(\mathcal{A},\mathcal{B})$ provides a homological orthogonal decomposition. A \emph{cotorsion torsion triple} $(\mathcal{C}, \mathcal{T}, \mathcal{F})$ in $\rmod R$ combines both structures: $(\mathcal{C}, \mathcal{T})$ is a cotorsion pair and $(\mathcal{T}, \mathcal{F})$ a torsion pair. Dually, a \emph{torsion cotorsion triple} $(\mathcal{T}, \mathcal{F}, \mathcal{D})$ in $\lmod R$ consists of a torsion pair $(\mathcal{T}, \mathcal{F})$ and a cotorsion pair $(\mathcal{F}, \mathcal{D})$ (see Definition \ref{general} for details).

Such triples were recently introduced by Bauer, Botnan, Oppermann, and Steen in abelian categories with enough projectives as a fundamental tool for developing representation theory of rectangular grids occurring in persistent homology theory, \cite{BBOS}. Our definition generalizes \cite[Definitions 2.6 and 2.9]{BBOS} (see also \cite{B}), where the cotorsion pairs involved were assumed to be \emph{complete}, that is, $\mathcal C$ and $\mathcal F$ were assumed to be special precovering classes, and $\mathcal T$ and $\mathcal D$ special preenveloping classes, cf.\ \cite[Remark 2.7]{BBOS}.

Our main contribution in this direction is a full classification of both the cotorsion torsion triples and the torsion cotorsion triples in module categories in Theorems \ref{ctt} and \ref{tct}. These results generalize and extend the earlier ones in \cite{B,BBOS}. The key difference in our approach is the use of infinite dimensional tilting theory in the sense of \cite{GT}, which makes it possible to remove the a priori assumption of completeness required in \cite{B,BBOS}.

For any ring $R$, cotorsion torsion triples in $\rmod R$ are classified by $1$-resolving subcategories of the category $\rfmod R$ of all finitely presented modules. The classic Ext-Tor duality then yields a monic map from the cotorsion torsion triples in $\rmod R$ to the dual setting of torsion cotorsion triples in $\lmod R$. This map is bijective when $R$ is left noetherian, and provides thus an explicit expression for the formal duality between the two settings (Corollary~\ref{explicit}). We also present alternative forms of the classification available in various particular settings in Corollaries~\ref{artin+noe} and~\ref{comnoe}, notably for artin algebras and commutative noetherian rings, and in Corollary~\ref{cor:GP-ctt} in the setting of Gorenstein homological algebra.

The structure of the paper is as follows. After presenting the necessary preliminaries in Section \ref{preliminaries}, Section \ref{Section3} is devoted to balanced pairs, their relations to cotorsion triples, and virtually Gorenstein rings. We characterize the conditions under which a balanced pair yields either a tilting or a cotilting cotorsion pair. Furthermore, for arbitrary rings, we prove that $\mathcal{GP}\times\mathcal{GI}$ is balanced if and only if $R$ is right virtually Gorenstein. Finally, Section \ref{Section4} is dedicated to classifying cotorsion torsion triples and torsion cotorsion triples in module categories using infinite dimensional tilting theory.

\section{Preliminaries}\label{preliminaries}

In what follows, $R$ denotes an associative ring with identity, $\rmod R$ the category of all (right $R$-) modules, and $\rfmod R$ the category of all finitely presented modules. The category of all left $R$-modules is denoted by $\lmod R$.

The class of all projective (injective, and flat) modules is denoted by $\mathcal P _0$ ($\mathcal I_0$, and $\mathcal F_0$), while $\mathcal P$ ($\mathcal I$) stands for the class of all modules of finite projective (injective) dimension. For a module $M$, its projective and injective dimensions are denoted by $\operatorname{pd}(M)$ and $\operatorname{id}(M)$, respectively.

The superscript $\perp_{i}$ denotes orthogonality with respect to the bifunctor $\textrm{Ext}_{R}^{i}(-,-)$ for $i\geq0$. More precisely, if $C\in{\rmod R}$, then $C^{\perp_{i}}=\textrm{Ker}~\textrm{Ext}_{R}^{i}(C,-)$ and $C ^{\perp_\infty} = \bigcap_{i \geq 1} C^{\perp_{i}}$, and analogously for $^{\perp_{i}}C$ and $^{\perp_\infty}C$. For a class $\mathcal {C}\subseteq{\rmod R}$, we let $\mathcal{C}^{\perp_{i}}=\bigcap_{C\in{\mathcal{C}}} C^{\perp_{i}}$, and similarly for $\mathcal{C}^{\perp_{\infty}}$, $^{\perp_{i}}\mathcal{C}$, and $^{\perp_{\infty}}\mathcal{C}$.

For a class $\mathcal {C}\subseteq{\rmod R}$, we will denote by $\mathcal C ^\intercal$ the class of all left $R$-modules $N$ such that $\textrm{Tor}^{R}_{1}(C,N) = 0$ for all $C \in \mathcal C$. Similarly, for a class $\mathcal {D}\subseteq{\lmod R}$, ${}^\intercal \mathcal D$ denotes the class of all modules $M$ such that $\textrm{Tor}^{R}_{1}(M,D) = 0$ for all $D \in \mathcal D$.

\subsection{Precovering and preenveloping classes}

Let $\mathcal{C}\subseteq \rmod R$ and $M\in{\rmod R}$. A homomorphism $\phi: M\rightarrow C$ in $\rmod R$ with $C\in \mathcal{C}$ is a
\emph{$\mathcal{C}$-preenvelope} of $M$ if for any
homomorphism $f:M\rightarrow C'$ with $C'\in \mathcal{C}$, there is
a homomorphism $g:C\rightarrow C'$ such that $g\phi=f$. Moreover, if
the only such $g$ are automorphisms of $C$ when $C'=C$ and $f=\phi$,
the $\mathcal{C}$-preenvelope $\phi$ is called a
\emph{$\mathcal{C}$-envelope} of $M$.

A $\mathcal{C}$-preenvelope $f:M\rightarrow C$ is \emph{special} if it is a monomorphism with $\textrm{Coker}f\in{{^{\perp_{1}}\mathcal{C}} }$. $\mathcal{C}\subseteq{\rmod R}$ is a \emph{(special) preenveloping class} if each module in $\rmod R$ has a (special) $\mathcal{C}$-preenvelope.

Dually, a homomorphism $\phi: C\rightarrow M$ in $\rmod R$ with $C\in \mathcal{C}$ is called a
\emph{$\mathcal{C}$-precover} of $M$ if for any
homomorphism $f:C'\rightarrow M$ with $C'\in \mathcal{C}$, there is
a homomorphism $g:C'\rightarrow C$ such that $f=\phi{g}$. Moreover, if
the only such $g$ are automorphisms of $C$ when $C'=C$ and $f=\phi$,
the $\mathcal{C}$-precover $\phi$ is called a
\emph{$\mathcal{C}$-cover} of $M$.

A $\mathcal{C}$-precover $f:C\rightarrow M$ is \emph{special} if it is an epimorphism with $\textrm{ker}f\in{\mathcal{C}^{\perp_{1}}}$. $\mathcal{C}\subseteq{\rmod R}$ is a \emph{(special) precovering class} if each module in $\rmod R$ has a (special) $\mathcal{C}$-precover.

We will need the following well-known closure properties of preenveloping and precovering classes (see e.g.\ \cite[Lemma 2.1]{ST}):

\begin{lemma}\label{easy}  Let $R$ be a ring and $\mathcal C$ be a class of modules closed under isomorphisms and direct summands.
\begin{enumerate}
\item[(1)] Assume $\mathcal C$ is preenveloping. Then $\mathcal C$ is closed under direct products.
\item[(2)] Assume $\mathcal C$ is precovering. Then $\mathcal C$ is closed under direct sums.
\end{enumerate}
\end{lemma}

\subsection{Cotorsion pairs and cotorsion triples}

A pair of classes $(\mathcal{A}, \mathcal{B})$ in $\rmod R$ is called a \emph{cotorsion pair} \cite[\S 5.2]{GT} (or a \emph{cotorsion theory}, \cite{Salce}) if $\mathcal{A} = {^{\perp_{1}}\mathcal{B}}$ and $\mathcal{B} = \mathcal{A}^{\perp_{1}}$. The cotorsion pair $(\mathcal{A}, \mathcal{B})$ is said to be \emph{complete} if $\mathcal{A}$ is special precovering in $\rmod R$. By Salce's
Lemma \cite[5.20]{GT}, this is equivalent to $\mathcal{B}$ being a special preenveloping class in $\rmod R$.

A cotorsion pair $\mathfrak C = (\mathcal{A}, \mathcal{B})$ is \emph{hereditary} if $\textrm{Ext}_R^i(A,B)=0$ for every $A\in{\mathcal{A}}$, $B\in{\mathcal{B}}$, and $i \ge 1$. By \cite[Lemma 5.24]{GT}, $\mathfrak C$ is hereditary if and only if $\mathcal{A}$ is \emph{resolving} in $\rmod R$, that is, $\mathcal{A}$ is closed under extensions, $\mathcal{P}_{0}\subseteq \mathcal{A}$,  and ${\mathcal{A}}$ is \emph{closed under kernels of epimorphisms} (i.e., $A\in{\mathcal{A}}$ whenever $0 \rightarrow A \rightarrow B \rightarrow C \rightarrow 0$ is an exact sequence in $\rmod R$ such that $B, C
\in{\mathcal{A}}$). Moreover, $\mathfrak C$ is hereditary if and only if $\mathcal{B}$ is \emph{coresolving} in $\rmod R$, that is, $\mathcal{B}$ is closed under extensions, $\mathcal{I}_{0}\subseteq\mathcal{B}$,  and $\mathcal{B}$ is \emph{closed under cokernels of monomorphisms} (i.e., $C\in{\mathcal{B}}$ whenever $0 \rightarrow A \rightarrow B \rightarrow C \rightarrow 0$ is an exact sequence in $\rmod R$ such that $A, B \in{\mathcal{B}}$).

A class $\mathcal C \subseteq \rmod R$ is \emph{thick} if $\mathcal C$ is closed under direct summands, extensions, kernels of epimorphisms, and cokernels of monomorphisms.

In Section \ref{Section4}, we will need $1$-resolving subcategories of finitely presented modules for a classification of cotorsion torsion triples:

\begin{definition}\label{small} Let $R$ be a ring and $\mathcal S \subseteq \rfmod R$.

$\mathcal S$ is a \emph{resolving subcategory} of $\rfmod R$ in case $\mathcal S \subseteq \rfmod R$ is closed under extensions, direct summands, $\mathcal{P}_{0} \cap \rfmod R \subseteq \mathcal{S}$, and $\mathcal S$ is closed under kernels of epimorphisms.

$\mathcal S$ is called \emph{$1$-resolving}, if moreover $\mathcal S$ consists of modules of projective dimension $\leq 1$.

Equivalently, $\mathcal S$ is a $1$-resolving subcategory of $\rfmod R$, if and only if $\mathcal S \subseteq \rfmod R$ consists of modules of projective dimension $\leq 1$, $R \in \mathcal S$, and $\mathcal S$ is closed under extensions and direct summands (see \cite[Lemma 13.48]{GT}).
\end{definition}

A cotorsion pair $(\mathcal{A}, \mathcal{B})$ is \emph{perfect} if every module $M$ admits both an $\mathcal{A}$-cover and a $\mathcal{B}$-envelope.

For example, $(\mathcal P _0, \rmod R)$, $(\rmod R,\mathcal I_0)$, and $(\mathcal F_0,\mathcal F_0 ^{\perp_1})$ are complete cotorsion pairs. The latter two cotorsion pairs are perfect for any ring $R$, while the first one is perfect, if and only if $R$ is a right perfect ring. Let $\mathcal{FPI} = \rfmod R ^{\perp_1}$ denote the class of all fp-injective (= absolutely pure) modules. Then the pair $(^{\perp_{1}}\mathcal{FPI}, \mathcal{FPI})$ is a complete cotorsion pair. It is hereditary if and only if $R$ is right coherent.

Let $\mathcal{A},\mathcal{B},\mathcal{C}\subseteq{\rmod R}$. The triple $t = (\mathcal{A},\mathcal{B},\mathcal{C})$ is called a \emph{cotorsion triple} \cite[\S3 of Chapter VI]{BR07} provided that both $(\mathcal{A},\mathcal{B})$ and $(\mathcal{B},\mathcal{C})$ are cotorsion pairs. The triple $t$ is \emph{complete} (\emph{hereditary}) provided that both these cotorsion pairs are complete (hereditary).

\begin{remark}\label{abund-rare} Complete hereditary cotorsion pairs are abundant: for example, if the ring $R$ is not right perfect, then there is a proper class of complete hereditary cotorsion pairs $(\mathcal{A}, \mathcal{B})$ in $\rmod R$ such that $\mathcal P_0 \subseteq \mathcal{A} \subseteq \mathcal F _0$, cf.\ \cite[2.1 and 2.2]{EGT}.

In contrast, cotorsion triples are rare: There is always the trivial cotorsion triple $t_0 = (\mathcal P_0, \rmod R, \mathcal I_0)$. However, if $t = (\mathcal{A},\mathcal{B},\mathcal{C})$ is a hereditary cotorsion triple in $\rmod R$, then the class $\mathcal B$ is coresolving and contains $\mathcal P_0$, whence $\mathcal B \supseteq \mathcal P$. So if $R$ has finite right global dimension, then $t = t_0$ is the only hereditary cotorsion triple in $\rmod R$.

For a different setting, assume $R$ is an \emph{Iwanaga-Gorenstein} ring, that is, $R$ is a left and right noetherian ring and the injective dimensions of both the regular left and right modules are finite (for example, let $R$ be any commutative local noetherian Gorenstein ring). If $R$ has infinite right global dimension, then there is a non-trivial complete hereditary cotorsion triple with the middle term $\mathcal B = \mathcal P = \mathcal I$ (see  e.g.\ \cite[Example 8.13]{GT}). We will have more on cotorsion triples and (virtually) Gorenstein rings in Section \ref{Section3}.
\end{remark}

\subsection{Tilting and cotilting modules and cotorsion pairs}

We now recall basic facts about (not necessarily finitely generated) tilting and cotilting modules.

A right $R$-module $T$ is  \emph{tilting} \cite{AC01,CT95} if the following conditions are satisfied:
\begin{enumerate}
\item[$(T1)$] $\operatorname{pd}(T)$ $<$ $\infty$.
\item[$(T2)$] $T^{(\lambda)}\in{T^{\perp_{\infty}}}$ for every cardinal $\lambda$.
\item[$(T3)$] There is a long exact sequence $$0\ra R\ra T_{0}\ra \cdots\ra T_{n}\ra 0$$ with $T_{i}$ $\in$ $\Add T$ for all $i \leq n$, where $n$ is the projective dimension of $T$. Here, $\Add T$ denotes the class of all direct summands of direct sums of copies of the module $T$.
\end{enumerate}

The class $\mathcal T _T = T^{\perp_{\infty}}$ is called the \emph{tilting class} induced by $T$, and the complete hereditary cotorsion pair $\mathfrak C _T = (^{\perp_1} \mathcal T _T, \mathcal T _T)$ the \emph{tilting cotorsion pair} induced by $T$.

If $n \geq 0$, then $T$, $\mathcal T _T$, and $\mathfrak C _T$ are called \emph{$n$-tilting} when $T$ has projective dimension $\leq n$; in this case, $^\perp \mathcal T _T$ coincides with the class of all modules possessing an $\Add T$-coresolution of length $\leq n$, and $\mathcal T _T$ with the class of all modules possessing a (possibly infinite) $\Add T$-resolution.

We will employ the following characterization of tilting cotorsion pairs from \cite[Corollary 13.20]{GT}:

\begin{lemma}\label{1tilt} Let $R$ be a ring and $\mathfrak C = (\mathcal A,\mathcal B)$ be a cotorsion pair in $\rmod R$.

Let $n \geq 0$. Then $\mathfrak C$ is $n$-tilting, if and only if the class $\mathcal A$ is resolving and consists of modules of projective dimension $\leq n$, and the class $\mathcal B$ is closed under direct sums.

In particular, $\mathfrak C$ is $1$-tilting, if and only if $\mathcal A$ consists of modules of projective dimension $\leq 1$ and $\mathcal B$ is closed under direct sums.
\end{lemma}

Dually, if $n \geq 0$, then $C \in \rmod R$ is \emph{$n$-cotilting} \cite{AC01,CT95} if
\begin{enumerate}
\item[$(C1)$] $\operatorname{id}(C) \leq n$.
\item[$(C2)$] $C^{\lambda}\in {^{\perp_{\infty}}C}$ for every cardinal $\lambda$.
\item[$(C3)$] There is a long exact sequence $$0\ra U_{n}\ra\cdots \ra U_{1}\ra U_{0}\ra E\ra 0$$ with $U_{i}$ $\in$ $\textrm{Prod}C$ for all $i \leq n$, and $E$ an injective cogenerator for $\rmod R$. Here Prod$C$ denotes the class of all direct summands of direct products of copies of the module $C$.
\end{enumerate}

Dually, we introduce the notions of the \emph{$n$-cotilting class} $\mathcal C _C = {}^{\perp_\infty} C$ and the complete hereditary \emph{$n$-cotilting cotorsion pair} $\mathfrak C _C = (\mathcal C _C,\mathcal C _C ^{\perp_1})$ induced by an $n$-cotilting module $C \in \rmod R$. Then \cite[Corollary 15.10]{GT} gives

\begin{lemma}\label{1cotilt} Let $R$ be a ring and $\mathfrak C = (\mathcal A,\mathcal B)$ be a cotorsion pair in $\rmod R$.

Let $n \geq 0$. Then $\mathfrak C$ is $n$-cotilting, if and only if the class $\mathcal B$ is coresolving and consists of modules of injective dimension $\leq n$, and the class $\mathcal A$ is closed under direct products.

In particular, $\mathfrak C$ is $1$-cotilting, if and only if $\mathcal B$ consists of modules of injective dimension $\leq 1$ and $\mathcal A$ is closed under direct products.
\end{lemma}

\begin{remark}\label{proofs}
Notice that the characterizations of tilting and cotilting cotorsion pairs in Lemmas \ref{1tilt} and \ref{1cotilt} do not assume their completeness. In both cases, completeness is a consequence of the closure properties of the classes $\mathcal A$ and $\mathcal B$.

Another remarkable fact is that while the statements of these lemmas are formally dual to each other, their proofs are very different. The proof of Lemma \ref{1tilt} employs set-theoretic homological algebra, while the proof of Lemma \ref{1cotilt} essentially entails showing the pure-injectivity of all cotilting modules.
\end{remark}

\subsection{Torsion pairs}

A pair of classes $(\mathcal{T}, \mathcal{F})$ in $\rmod R$ is called a \emph{torsion pair} \cite{Dickson} if the following conditions are satisfied:

\begin{enumerate}
\item[(1)] $\textrm{Hom}_{R}(T,F)=0$ for all $T\in{\mathcal{T}}$ and $F\in{\mathcal{F}}$;
\item[(2)] For each $M\in{\rmod R}$, there is a short exact sequence in $\rmod R$
$$0\ra T\ra M\ra F\ra0$$
with $T\in{\mathcal{T}}$ and $F\in{\mathcal{F}}$.
\end{enumerate}
In this case, we call $\mathcal{T}$ a \emph{torsion class} and $\mathcal{F}$ a \emph{torsion-free class}. In any torsion pair $(\mathcal{T}, \mathcal{F})$, the two classes determine each other; specifically,
\begin{center}
$\mathcal{T}^{\perp_{0}}=\mathcal{F}$ and $\mathcal{T}={^{\perp_0}\mathcal{F}}$.
\end{center}
Furthermore, a class $\mathcal{C}\subseteq{\rmod R}$ is a torsion class if and only if it is closed
under quotients, extensions and direct sums. Dually, $\mathcal{C}$ is a torsion-free class if and only if it is closed
under submodules, extensions and direct products.

\begin{example}\label{tiltcotilt-tpairs} If $T$ is any $1$-tilting module, then the $1$-tilting class $\mathcal T _T$ induced by $T$ is a torsion class in $\rmod R$. Dually, if $C \in \rmod R$ is $1$-cotilting, then the $1$-cotilting class induced by $C$ is a torsion-free class in $\rmod R$, see \cite[Lemmas 14.2 and 15.21]{GT}.
\end{example}

\subsection{Gorenstein modules}

Recall that a module $M$ is \emph{Gorenstein projective} \cite{EJ95} if there is an  exact sequence $$\mathbf{P}:\cdots\ra P_1\ra P_0\ra P^0\ra P^1\ra \cdots$$ of projective right $R$-modules such that $M$ $\cong$ $\ker(P^{0}\ra P^1)$ and $\textrm{Hom}_R(\mathbf{P},Q)$ is exact for every projective module $Q$. The class of all Gorenstein projective modules will be denoted by $\mathcal{GP}$.

Note that $\mathcal{GP} \supseteq \mathcal P_0$, and $\mathcal{GP}^{\perp_\infty} = \mathcal{GP}^{\perp_1} \supseteq \mathcal P_0$, whence $\mathcal{GP}^{\perp_1} \supseteq \mathcal P$ and $\mathcal{GP} \cap \mathcal P = \mathcal P_0$.

The notion of a \emph{Gorenstein injective module} is defined dually. The class of all Gorenstein injective modules will be denoted by $\mathcal{GI}$. Dually, we have $\mathcal{GI} \supseteq \mathcal I_0$, and $^{\perp_\infty}\mathcal{GI} = {}^{\perp_1}\mathcal{GI} \supseteq \mathcal I_0$, whence $^{\perp_1}\mathcal{GI} \supseteq \mathcal I$ and $\mathcal{GI} \cap \mathcal I = \mathcal I_0$.

We will make use of the following properties of the classes $\mathcal{GP}$ and $\mathcal{GI}$.

\begin{lemma}\label{cort-sar-stov} Let $R$ be a ring.

\begin{enumerate}
\item[(1)] $(\mathcal{GP},\mathcal{GP}^{\perp_1})$ is a hereditary cotorsion pair such that the class $\mathcal{GP}^{\perp_1}$ is resolving.
\item[(2)] $(^{\perp_1}\mathcal{GI},\mathcal{GI})$ is a complete perfect hereditary cotorsion pair such that the class $^{\perp_1}\mathcal{GI}$ is coresolving and closed under direct limits.
\end{enumerate}
\end{lemma}
\begin{proof} (1) By \cite[Corollary 3.4(1)]{CS}.

(2) By \cite[Lemmas 5.1 and 5.4, and Theorem 5.6]{SaS}.
\end{proof}

\begin{remark}\label{independent} Completeness of the hereditary cotorsion pairs $(\mathcal{GP},\mathcal{GP}^{\perp_1})$ for arbitrary rings has so far been proven only under additional set-theoretic assumptions of the existence of a supercompact cardinal $\geq \card{R}$, or the Vop\v{e}nka Principle, cf.\ \cite[Theorem 5.2 and Proposition 5.6]{CS} and \cite[Theorem 1.4(B)]{C}, respectively. 

In the settings considered in the present paper, the class $\mathcal{GP}^{\perp_1}$ will be closed under unions of well-ordered chains, whence the completeness of $(\mathcal{GP},\mathcal{GP}^{\perp_1})$ will follow by the Countable Telescope Conjecture proved in \cite[Theorem 3.5]{SaSt}.
\end{remark}

\medskip
A right $R$-module $M$ is \emph{Gorenstein flat} \cite{EJT93} if there is an exact sequence  $$\mathbf{F}:\cdots\ra F_1\ra F_0\ra F^0\ra F^1\ra \cdots$$ of flat right $R$-modules with $M\cong \textrm{ker}(F^{0}\ra F^{1})$  such that $\mathbf{F}\otimes_{R}E$ is exact for every injective left $R$-module $E$. We will denote by $\mathcal{GF}$ the class of all Gorenstein flat modules.

Following \cite{DLW,WE25}, we will call a ring $R$ \emph{right virtually Gorenstein}, if $\mathcal{GP}^{\perp_{1}} ={^{\perp_{1}}\mathcal{GI}}$. By \cite[Proposition 3.16]{WE25}, if $R$ is right virtually Gorenstein, then the class $\mathcal{GP}$ agrees with the class of all projectively coresolved Gorenstein flat right $R$-modules introduced in \cite{SaS}.

\begin{example}\label{vG} (1) Let $R$ be an Iwanaga-Gorenstein ring. Then $R$ is right virtually Gorenstein. In particular, any commutative local noetherian Gorenstein ring is virtually Gorenstein.

Indeed, there are the cotorsion pairs $(\mathcal{GP},\mathcal P)$ and $(\mathcal P,\mathcal{GI})$, and hence the cotorsion triple $(\mathcal{GP},\mathcal P,\mathcal{GI})$, in $\rmod R$. Moreover, if $R$ is \emph{$n$-Iwanaga-Gorenstein} (that is, if the injective dimension of $R$ equals $n$), then the second cotorsion pair is $n$-tilting, see \cite[Example 8.13 and Theorem 17.12]{GT}.

(2) For each $0 \leq n < \infty$, there exist commutative noetherian virtually Gorenstein rings of Krull dimension $n$ that are not Iwanaga-Gorenstein: In \cite[Example 2.15]{DLW}, a commutative local artinian virtually Gorenstein ring $R_0$ has been constructed such that $\operatorname{id}(R_0) = \infty$ (so $R_0$ is not Iwanaga-Gorenstein) and $\rmod {R_0} = \mathcal{GP}^{\perp_{1}} ={^{\perp_{1}}\mathcal{GI}}$. For example, for any field $k$, the ring $R_0 = k[x,y]/(x^2,xy,y^2)$ has these properties.

Let $1 \leq n < \infty$. If $R_n$ is any commutative noetherian local $n$-Gorenstein ring, then the ring direct product $R_0 \times R_n$ is virtually Gorenstein of Krull dimension $n$, but it is not Gorenstein.

(3) By \cite[Proposition 4.3]{BK08}, for any field $k$, the commutative artinian ring $S_0 = k[x,y,z]/(x^2,yz,y^2-xz,z^2-xy)$ is not virtually Gorenstein. So if $1 \leq n \leq \infty$ and $S_n$ is any commutative noetherian ring of Krull dimension $n$, then the ring direct product $S_0 \times S_n$ has Krull dimension $n$, but it is not virtually Gorenstein.
\end{example}

\medskip
We refer to \cite{EJ} and \cite{GT} for further unexplained terminology and properties of the notions defined above.

\section{Balanced pairs and virtually Gorenstein rings}\label{Section3}

Let $\mathcal{X}\subseteq \rmod R$ and $M\in{\rmod R}$. An \emph{$\mathcal{X}$-resolution} $X_{\bullet}\rightarrow M$ of $M$ is a (not necessarily exact) complex $$\cdots \rightarrow X_1\rightarrow X_0\rightarrow M\rightarrow 0$$ with each $X_i\in \mathcal{X}$, which is exact when applying ${\rm Hom}_R(X,-)$ for every $X\in \mathcal{X}$; see \cite[Definition 8.1.2]{EJ}. In this case, we will say that the complex $X_{\bullet}\rightarrow M$ is \emph{${\rm Hom}_R(\mathcal{X},-)$-acyclic}. Dually, one has the notion of \emph{$\mathcal{X}$-coresolution} $M\rightarrow X^{\bullet}$ of $M$.

Recall \cite[8.1.3]{EJ} that if the class $\mathcal X$ is precovering (preenveloping), then each module possesses an $\mathcal X$-resolution ($\mathcal X$-coresolution).

The notion of a balanced pair goes back to Chen \cite[Definition 1.1]{Chen10}:

\begin{definition}\label{balanced} Let $R$ be a ring and $\mathcal F \times \mathcal G$ be a pair of classes of modules. Then $\mathcal F \times \mathcal G$ is \emph{balanced} provided that
\begin{enumerate}
\item[(1)] $\mathcal F$ is precovering and $\mathcal G$ is preenveloping.
\item[(2)] For each module $M$, there is an $\mathcal F$-resolution $F_{\bullet}\rightarrow M$ which is $\textrm{Hom}_R(-,\mathcal G)$-acyclic.
\item[(3)] For each module $N$, there is a $\mathcal G$-coresolution $N\rightarrow G^{\bullet}$ which is $\textrm{Hom}_R(\mathcal{F},-)$-acyclic.
\end{enumerate}
The balanced pair $\mathcal F \times \mathcal G$ is \emph{special} in case the class $\mathcal F$ is special precovering and $\mathcal G$ special preenveloping.
\end{definition}

\begin{example}\label{basic} (1) Let $\mathcal E$ denote the class of all free modules. Then the pairs $\mathcal P_0 \times \mathcal I _0$ and $\mathcal E \times \mathcal I _0$ are special balanced for any ring $R$.

(2) If $R$ is an Iwanaga-Gorenstein ring, then the pair $\mathcal{GP} \times \mathcal{GI}$ is special balanced by \cite[Lemmas 12.1.2 and 12.1.3]{EJ}.

(3) Let $R$ be any ring, and $\mathcal{PP}$ and $\mathcal{PI}$ denote the classes of all pure-projective and pure-injective modules, respectively. Then $\mathcal{PP}\times\mathcal{PI}$ is a balanced pair. However, whether this pair is special depends on the structure of the ring $R$:
 
If $R$ is von Neumann regular, then $\mathcal{PP} = \mathcal P _0$ and $\mathcal{PI} = \mathcal I _0$, so the pair $\mathcal{PP}\times\mathcal{PI}$ is special by part (1). However, if $R$ is right noetherian, then the balanced pair $\mathcal{PP}\times\mathcal{PI}$ is special, if and only if $R$ is a \emph{right pure-semisimple ring}, that is, all modules are pure injective.

Indeed, since $\mathcal F_0 = {}^{\perp_1} \mathcal{PI}$ by \cite[5.18(i)]{GT} and $\mathcal{PP}{}^{\perp_{1}} = \mathcal{FPI}$, Lemma \ref{filt-cofilt} below implies that if $\mathcal{PP}\times\mathcal{PI}$ is special, then (a) $\mathcal{PP}={}^{\perp_{1}}\mathcal{FPI}$ and (b) $\mathcal{PI}=\mathcal F_0 ^{\perp_1}$. Since $R$ is right noetherian, (a) implies $\mathcal{PP}=\rmod R$, hence each pure-exact sequence splits, and all modules are pure-injective. Conversely, if $R$ is right pure-semisimple, then trivially $\mathcal{PI}=\rmod R$, and also $\mathcal{PP}=\rmod R$, because each module is a direct limit, hence a (split) pure-epimorphic image of a direct sum, of finitely presented modules.
\end{example}

\begin{remark}\label{R1} The key property of balanced pairs is that they make it possible to compute cohomology in two equivalent ways, using $\mathcal F$-resolutions or $\mathcal G$-coresolutions (see \cite[Theorem 8.2.14]{EJ}). Indeed, conditions (2) and (3) of Definition \ref{balanced} say that the functor Hom$(-,-)$ is \emph{right balanced} by $\mathcal F \times \mathcal G$ in the sense of \cite[Definition 8.2.13]{EJ}.
\end{remark}

We will need the following property of balanced pairs proved in \cite[Lemma 3.1]{EPZ}:

\begin{lemma}\label{enochs} Let $\mathcal F \times \mathcal G$ be a balanced pair. A short exact sequence in $\rmod R$ is $\Hom RF{-}$-exact for each $F \in \mathcal F$ if and only if it is $\Hom R{-}G$-exact for each $G \in \mathcal G$.
\end{lemma}

Next, we will examine how balanced pairs behave with respect to usual closure properties studied in homological algebra. We start with a lemma that extends Lemma \ref{easy} from direct products of modules in $\mathcal C$ to arbitrary $\mathcal C$-cofiltered modules, and from direct sums of modules in $\mathcal C$ to arbitrary $\mathcal C$-filtered modules.

Recall that if $\mathcal C$ is a class of modules and $M \in \rmod R$, then $M$ is \emph{$\mathcal C$-filtered} (or a \emph{transfinite extension} of modules from $\mathcal C$) provided that there exists an increasing chain $( M_\alpha \mid \alpha \leq \sigma )$ consisting of submodules of $M$ such that $M_0 = 0$, $M_{\alpha + 1}/M_\alpha$ is isomorphic to an element of $\mathcal C$ for each $\alpha < \sigma$, $M_\alpha = \bigcup_{\beta < \alpha} M_\beta$ for each limit ordinal $\alpha \leq \sigma$, and $M_\sigma = M$. A class $\mathcal C$ is called \emph{filtration closed} if it contains all $\mathcal C$-filtered modules.

Dually, a module $M$ is \emph{$\mathcal C$-cofiltered} provided that there exists a chain $( M_\alpha \mid \alpha \leq \sigma )$ of modules and epimorphisms $\pi_\alpha : M_{\alpha + 1} \to M_\alpha$ ($\alpha < \sigma$) such that $M_0 = 0$, Ker$(\pi_\alpha)$ is isomorphic to an element of $\mathcal C$ for each $\alpha < \sigma$, $M_\alpha = \varprojlim_{\beta < \alpha} M_\beta$ for each limit ordinal $\alpha \leq \sigma$, and $M_\sigma = M$. A class $\mathcal C$ is called \emph{cofiltration closed} if it contains all $\mathcal C$-cofiltered modules.

For example, if $R = \mathbb Z$ and $\mathcal C = \{ \mathbb Z _p \}$ for a prime $p$, then the Pr\"{u}fer group $\mathbb Z _{p^\infty}$ is $\mathcal C$-filtered, while the $p$-adic group $\mathbb J _p$ is $\mathcal C$-cofiltered.

\begin{lemma}\label{filt-cofilt}  Let $R$ be a ring and $\mathcal C$ be a class of modules closed under direct summands.
\begin{enumerate}
\item[(1)]  Assume $\mathcal C$ is special preenveloping. Then $\mathcal C$ is the right-hand class of a complete cotorsion pair, hence it is cofiltration closed.
\item[(2)] Assume $\mathcal C$ is special precovering. Then $\mathcal C$ is the left-hand class of a complete cotorsion pair, hence it is filtration closed.
\end{enumerate}
\end{lemma}
\begin{proof}(1) Let $X \in  ({^{\perp_{1}}} \mathcal{C}){^{\perp_{1}}}$. By assumption, there exists a short exact sequence  $0 \to X \to Y \to Z \to 0$  with $Y \in \mathcal{C}$ and $Z \in {^{\perp_1} \mathcal{C}}$.  Since $X \in {(^{\perp_1} \mathcal{C})^{\perp_1}}$, the sequence splits, whence $X \in {\mathcal{C}}$ because $\mathcal{C}$ is closed under direct summands. So ($^{\perp_1} \mathcal{C}$, $\mathcal{C}$) is a complete cotorsion pair. By the Lukas Lemma \cite[Lemma 6.37]{GT}, $\mathcal{C}$ is cofiltration closed.

(2) The proof is dual to (1), using Eklof Lemma \cite[Lemma 6.2]{GT} in place of the Lukas one.
\end{proof}

Now, we can characterize the balanced pairs that induce hereditary cotorsion triples.

\begin{proposition}\label{balanced-hered} 
Let $R$ be a ring and $\mathcal F\times\mathcal G$ be a balanced pair in
$\rmod R$. Then the following conditions are equivalent:
\begin{enumerate}
\item[(1)] The class $\mathcal F^{\perp_1}$ is closed under cokernels of monomorphisms and contains $\mathcal P_0$, and 
the class ${}^{\perp_1}\mathcal G$ is closed under kernels of epimorphisms and contains $\mathcal I_0$. 
\item[(2)] There is a hereditary cotorsion triple $(\mathcal A,\mathcal C, \mathcal B)$ in $\rmod R$ such that $\mathcal C = \mathcal F^{\perp_1} = {}^{\perp_1}\mathcal G$.
\end{enumerate}
\end{proposition}
\begin{proof} $(1)\Rightarrow(2)$. Both $\mathcal U = \mathcal F^{\perp_1}$ and $\mathcal V = {}^{\perp_1}\mathcal G$ are closed under extensions and direct summands. Also, $\mathcal U$ contains all injective modules, and $\mathcal V$ all projective modules. Thus the assumptions imply that $\mathcal U$ is coresolving and $\mathcal V$ is resolving. So $\mathcal U = \mathcal F^{\perp_\infty}$ and $\mathcal V = {}^{\perp_\infty}\mathcal G$.

Let $M\in\mathcal V$ and take a projective presentation $0\to K\to P\to M\to0$. Since $M\in{}^{\perp_1}\mathcal G$, the sequence is $\Hom{R}{-}{\mathcal G}$-exact, hence $\Hom{R}{\mathcal F}{-}$-exact by Lemma~\ref{enochs}. As $P\in\mathcal U$, the long exact sequence of Ext gives $K\in\mathcal U$. Since $\mathcal U$ is closed under cokernels of monomorphisms, it follows that $M\in\mathcal U$.

Conversely, take $M\in\mathcal U$ and an exact sequence $0\to M\to I\to L\to 0$ with $I$ injective. This sequence is $\Hom{R}{\mathcal F}{-}$-exact, and hence $\Hom{R}{-}{\mathcal G}$-exact. As $I\in\mathcal V$, the long exact sequence gives $L\in\mathcal V$. Closure under kernels of epimorphisms then gives $M\in\mathcal V$. 

Consequently, $\mathcal U=\mathcal V$, and (2) holds.

The implication $(2)\Rightarrow(1)$ is clear.
\end{proof}
 
Our next proposition proves several additional properties of general hereditary cotorsion triples, notably their completeness:

\begin{proposition}\label{balanced-structural}
Let $R$ be a ring and $(\mathcal A,\mathcal C, \mathcal B)$ be a hereditary cotorsion triple in $\rmod R$. Then the class $\mathcal C$ is thick and definable, $\mathcal P \cup \mathcal I \subseteq \mathcal C$, $\mathcal A \cap \mathcal C = \mathcal P_0$, and $\mathcal B \cap \mathcal C = \mathcal I_0$. 

Moreover, the cotorsion pair $(\mathcal A,\mathcal C)$ is generated by a set of strongly countably presented modules, the cotorsion pair $(\mathcal C,\mathcal B)$ is perfect, and it is generated by a set of $\leq \lambda$-presented modules where $\lambda = |R| + \aleph_0$. In particular, the cotorsion triple $(\mathcal A,\mathcal C, \mathcal B)$ is complete.  
\end{proposition}
\begin{proof} 
Since the class $\mathcal C$ is both resolving and coresolving and contains $\mathcal P _0 \cup \mathcal I_0$, it also contains $\mathcal P \cup \mathcal I$.

As $\mathcal C$ is also closed under direct summands, $\mathcal C$ is thick. Being a left class of a cotorsion pair, $\mathcal C$ is filtration closed by \cite[Eklof Lemma 6.2]{GT}. So by \cite[Lemma 5.1]{SaS}, $\mathcal C$ is closed under direct limits. By \cite[Theorem 3.5]{SaSt}, the hereditary cotorsion pair $(\mathcal A,\mathcal C)$ is generated by a set of strongly countably presented modules, hence it is complete, and the class $\mathcal C$ is definable. 

Let $X\in\mathcal A\cap\mathcal U$ and choose a short exact sequence $0\to K\to P\to X\to0$ with $P$ projective. Since $\mathcal C$ is resolving, $K\in\mathcal C$, so $\Ext{1}{R}{X}{K}=0$ and $X$ is projective. Dually, if $Y\in\mathcal B\cap\mathcal C$, take $0\to Y\to I\to L\to 0$ with $I$ injective. Then $L\in\mathcal C$, and the sequence splits since $\Ext{1}{R}{L}{Y}=0$. This proves that $\mathcal A\cap\mathcal C=\mathcal P_0$ and $\mathcal B\cap\mathcal C=\mathcal I_0$. 

Since $\mathcal C$ is definable, it is closed under pure submodules and pure quotients. Thus, purification \cite[Lemma 2.25(a)]{GT} can be used to prove that each module in the class $\mathcal C$ is $\mathcal C ^{\leq \lambda}$-filtered where $\mathcal C ^{\leq \lambda}$ is the class of all $\leq \lambda$-presented modules in $\mathcal C$. Since $\mathcal C$ is closed under direct limits, the cotorsion pair $(\mathcal C,\mathcal B)$ is perfect by \cite[Corollary 5.32]{GT}.
\end{proof}

Recall that a module is \emph{Ding projective} \cite{DLM09,Gil10} if it is a cycle of an exact complex $P$ of projective modules such that $\Hom{R}{P}{F}$ is exact for each flat module $F$. Dually, a module is called \emph{Ding injective} \cite{DM08,Gil10} if it is a cycle of an exact complex $I$ of injective modules such that $\Hom{R}{E}{I}$ is exact for each fp-injective module $E$. We will denote these classes by $\mathcal{DP}$ and $\mathcal{DI}$, respectively. Here fp-injective means absolutely pure, as recalled in Section~2.

Let $\mathcal{PGF}$ denote the class of all projectively coresolved Gorenstein flat modules introduced in \cite{SaS}. These are the cycles of exact complexes $P$ of projective modules such that $P\otimes_R E$ is exact for each injective left $R$-module $E$. 

\begin{corollary}\label{balanced-structural-gorenstein}
In the notation and under the hypotheses of
Proposition~\ref{balanced-hered}(1), we have
$\mathcal F\subseteq\mathcal A\subseteq\mathcal{PGF}
\subseteq\mathcal{DP}\subseteq\mathcal{GP}$ and
$\mathcal G\subseteq\mathcal B\subseteq\mathcal{DI}
\subseteq\mathcal{GI}$.
Consequently, $\mathcal{PGF}^{\perp_1}\subseteq\mathcal C$ and
${}^{\perp_1}\mathcal{DI}\subseteq\mathcal C$.
\end{corollary}

\begin{proof}
The inclusions $\mathcal F\subseteq\mathcal A$ and
$\mathcal G\subseteq\mathcal B$ follow directly from the orthogonal
descriptions of $\mathcal U$. We prove the stronger inclusions for
$\mathcal A$ and $\mathcal B$.

Take $A\in\mathcal A$. A special $\mathcal U$-preenvelope gives an
exact sequence $0\to A\to C^0\to A^1\to0$ with
$C^0\in\mathcal U$ and $A^1\in\mathcal A$. Since $\mathcal A$ is
extension closed, $C^0\in\mathcal A\cap\mathcal U=\mathcal P_0$.
Iterating this construction and splicing a projective
resolution of $A$ gives an exact complex $P$ of projectives
whose cycles all belong to $\mathcal A$, with $A$ one of its cycles.
For every $C\in\mathcal U$, the complex
$\Hom{R}{P}{C}$ is exact.

Since $\mathcal U$ contains all projectives and is closed under direct
limits, it contains all flat modules. Let $\langle\mathcal F_0\rangle$ denote the smallest definable class
containing all flat modules. Then $\langle\mathcal F_0\rangle\subseteq\mathcal U$.
It follows from \cite[Corollary 4.5]{SaS} that $A\in\mathcal{PGF}$. This implies that $\mathcal A\subseteq\mathcal{PGF}$.
Applying \cite[Corollary 4.5]{SaS} again yields $\mathcal{PGF}\subseteq\mathcal{DP}$,
and $\mathcal{DP}\subseteq\mathcal{GP}$ follows from the definitions.

For $B\in\mathcal B$, a special $\mathcal U$-precover gives an exact
sequence $0\to B_1\to C_0\to B\to0$ with
$B_1\in\mathcal B$ and $C_0\in\mathcal U$. By extension closure,
$C_0\in\mathcal U\cap\mathcal B=\mathcal I_0$. Iterating and
splicing an injective coresolution produces an exact
complex $I$ of injectives whose cycles all belong to $\mathcal B$ and
one of whose cycles is $B$.

Every fp-injective module is a pure submodule of an injective module.
Since $\mathcal U$ contains all injectives and is closed under pure
submodules, $\mathcal{FPI}\subseteq\mathcal U$. Thus for every
$E\in\mathcal{FPI}$ and every cycle $Z$ of $I$, we have
$\Ext{1}{R}{E}{Z}=0$. Therefore
$\Hom{R}{E}{I}$ is exact and $B$ is Ding injective. This implies that $\mathcal B\subseteq\mathcal{DI}$.
The inclusion $\mathcal{DI}\subseteq\mathcal{GI}$ follows from the
definitions. Finally, reverse the inclusions under Ext-orthogonality,
using $\mathcal A^{\perp_1}=\mathcal U={}^{\perp_1}\mathcal B$.
\end{proof}

A consequence of Corollary \ref{balanced-structural-gorenstein} yields the following result.
\begin{corollary}\label{ding-balance-pgf} If $\mathcal{DP}\times\mathcal{DI}$ is balanced in $\rmod R$, then $\mathcal{DP}=\mathcal{PGF}$. In particular, every Ding projective module is Gorenstein flat.
\end{corollary}
\begin{proof}
The class $\mathcal{DP}$ is resolving and $\mathcal{DI}$ is coresolving. Since projective modules are flat and injective modules are fp-injective, the defining acyclicity conditions give $\mathcal P_0\subseteq\mathcal{DP}^{\perp_1}$ and $\mathcal I_0\subseteq{}^{\perp_1}\mathcal{DI}$. Dimension shifting verifies the remaining hypotheses of Proposition~\ref{balanced-hered}(1). It follows from Corollary~\ref{balanced-structural-gorenstein} that $\mathcal{DP}\subseteq\mathcal{PGF}\subseteq\mathcal{DP}$. This shows $\mathcal{DP}\subseteq\mathcal{GF}$. 
\end{proof}

In the setting of Proposition~\ref{balanced-hered}(1), the hereditary cotorsion triple induced by a balanced pair $\mathcal F \times \mathcal G$ is determined
by the common orthogonal class $\mathcal C=\mathcal F^{\perp_1}={}^{\perp_1}\mathcal G$. Its end terms satisfy $\mathcal F\subseteq\mathcal A={}^{\perp_1}\mathcal C$ and
$\mathcal G\subseteq\mathcal B=\mathcal C^{\perp_1}$. Both these inclusions can be strict:
 
\begin{example}\label{strict} Let $R$ be a $0$-Gorenstein commutative artinian ring which is not pure-semisimple (for example, let $R = k[x,y]/(x^2,y^2)$ where $k$ is a field). By Example \ref{basic}(3), $\mathcal{PP}\times\mathcal{PI}$ is a balanced pair, and $\mathcal{PP}^{\perp_1} = \mathcal I_0 = \mathcal C = \mathcal F _0 = {}^{\perp_1}\mathcal{PI}$, so the pair fits the setting of Proposition \ref{balanced-hered}(1). However, the corresponding complete cotorsion triple is $(\rmod R,\mathcal C,\rmod R)$.
\end{example}

The following one-sided speciality criterion gives conditions under which both equalities do hold. 

\begin{theorem}\label{hered}
Let $\mathcal F$ and $\mathcal G$ be classes of modules closed under direct summands. Then the following conditions are equivalent:
\begin{enumerate}
\item[(1)] The classes $\mathcal F$ and $\mathcal G$ fit in a complete hereditary cotorsion triple $(\mathcal F,\mathcal H,\mathcal G)$.
\item[(2)] $\mathcal F\times\mathcal G$ is a special balanced pair, $\mathcal F$ is closed under kernels of epimorphisms, $\mathcal G$ is closed under cokernels of monomorphisms, $\mathcal P_0\subseteq\mathcal F^{\perp_1}$, and $\mathcal I_0\subseteq{}^{\perp_1}\mathcal G$.
\item[(3)] $\mathcal F\times\mathcal G$ is a balanced pair, $\mathcal F$ is special precovering and closed under kernels of epimorphisms, $\mathcal G$ is closed under cokernels of monomorphisms, $\mathcal P_0\subseteq\mathcal F^{\perp_1}$, and $\mathcal I_0\subseteq{}^{\perp_1}\mathcal G$.
\item[(4)] $\mathcal F\times\mathcal G$ is a balanced pair, $\mathcal F$ is closed under kernels of epimorphisms, $\mathcal G$ is special preenveloping and closed under cokernels of monomorphisms, $\mathcal P_0\subseteq\mathcal F^{\perp_1}$, and $\mathcal I_0\subseteq{}^{\perp_1}\mathcal G$.
\end{enumerate}
In this case $\mathcal H=\mathcal F^{\perp_1}={}^{\perp_1}\mathcal G$ is definable, $(\mathcal F,\mathcal H)$ is generated by a set of strongly countably presented modules, and $(\mathcal H,\mathcal G)$ is perfect.
\end{theorem}
\begin{proof}
The implication $(1)\Rightarrow(2)$ was proved by Enochs, Jenda, Torrecillas and Xu; a published proof appears in \cite[Proposition 2.6]{Chen10}. Clearly, (2) implies both (3) and (4). We prove that either of the latter conditions implies (1).

First, under either (3) or (4), we claim that $\mathcal P_0\subseteq\mathcal F$ and $\mathcal I_0\subseteq\mathcal G$. Indeed, in case (3), special $\mathcal F$-precovers are epic and split
at projective modules, so $\mathcal P_0\subseteq\mathcal F$. Since $R\in\mathcal F$, the balanced $\mathcal G$-coresolution
of an injective module begins with a split monomorphism into an
object of $\mathcal G$. Hence
$\mathcal I_0\subseteq\mathcal G$. Under (4), the dual argument applies: special
$\mathcal G$-preenvelopes give
$\mathcal I_0\subseteq\mathcal G$, and evaluating
$\mathcal F$-resolutions at an injective cogenerator
$E\in\mathcal G$ yields epimorphic precovers and
$\mathcal P_0\subseteq\mathcal F$.

Since $\mathcal P_0\subseteq\mathcal F$ and $\mathcal F$ is closed under kernels of epimorphisms, all syzygies of modules in $\mathcal F$ belong to $\mathcal F$. Dimension shifting therefore shows that $\mathcal F^{\perp_1}$ is coresolving. Dually, $\mathcal I_0\subseteq\mathcal G$ and closure under cokernels of monomorphisms imply that ${}^{\perp_1}\mathcal G$ is resolving. Propositions~\ref{balanced-hered} and \ref{balanced-structural} apply and give
$$
 \mathcal H=\mathcal F^{\perp_1}={}^{\perp_1}\mathcal G,\quad
 (\mathcal A,\mathcal H,\mathcal B)
 =({}^{\perp_1}\mathcal H,\mathcal H,\mathcal H^{\perp_1}),
$$
where the latter triple is complete and hereditary.

Under (3), Lemma~\ref{filt-cofilt} gives the cotorsion pair $(\mathcal F,\mathcal H)$, so $\mathcal F=\mathcal A$. The complete hereditary triple above yields the special balanced pair $\mathcal F\times\mathcal B$ by \cite[Proposition 2.6]{Chen10}. This implies that both $\mathcal F\times\mathcal G$ and
$\mathcal F\times\mathcal B$ are admissible balanced pairs with epimorphic $\mathcal F$-precovers, and therefore  we have $\mathcal G=\mathcal B$ by \cite[Proposition~3.4]{EPZ}.
Under (4), Lemma~\ref{filt-cofilt} instead gives $\mathcal G=\mathcal B$, and the dual uniqueness argument gives $\mathcal F=\mathcal A$. This proves (1). The remaining assertions follow from Propositions~\ref{balanced-hered} and \ref{balanced-structural}.
\end{proof}

\begin{remark}\label{rem:salce-type}
Theorem~\ref{hered} gives a Salce-type correspondence for balanced pairs, c.f. \cite[Salce Lemma 5.20]{GT}.
Under its closure and orthogonality hypotheses,
$\mathcal F$ is special precovering if and only if
$\mathcal G$ is special preenveloping. In either case,
$(\mathcal F,\mathcal F^{\perp_1},\mathcal G)$
is a complete hereditary cotorsion triple, with
$\mathcal F^{\perp_1}={}^{\perp_1}\mathcal G$.
\end{remark}
For the particular case of the classes $\mathcal{GP}$ and $\mathcal{GI}$, we obtain
\begin{corollary}\label{virtGor}\label{cor-GP-virtually-Gor}
Let $R$ be any ring. Then the following conditions are equivalent:
\begin{enumerate}
\item[(1)] $R$ is right virtually Gorenstein.
\item[(2)] $(\mathcal{GP},\mathcal{GP}^{\perp_1},\mathcal{GI})$ is a complete hereditary cotorsion triple in $\rmod R$.
\item[(3)] $\mathcal{GP}\times\mathcal{GI}$ is a special balanced pair in $\rmod R$.
\item[(4)] $\mathcal{GP}\times\mathcal{GI}$ is a balanced pair in $\rmod R$.
\end{enumerate}
\end{corollary}
\begin{proof}
$(1)\Rightarrow(2)$. In view of parts (1) and (2) of Lemma \ref{cort-sar-stov}, and of Theorem \ref{hered}, we only have to prove that the class $\mathcal{GP}$ is special precovering. But this follows from \cite[Theorem 3.5]{SaSt}.

$(2)\Rightarrow(3)$ follows from Lemma~\ref{hered}, and $(3)\Rightarrow(4)$ is immediate from the definition.

$(4)\Rightarrow(1)$. The class $\mathcal{GP}$ is resolving and $\mathcal{GI}$ is coresolving, so dimension shifting shows that $\mathcal{GP}^{\perp_1}$ is coresolving and ${}^{\perp_1}\mathcal{GI}$ is resolving. Observe that $\mathcal P_0\subseteq\mathcal{GP}^{\perp_1}$ and $\mathcal I_0\subseteq{}^{\perp_1}\mathcal{GI}$. Theorem~\ref{hered} now gives $\mathcal{GP}^{\perp_1}={}^{\perp_1}\mathcal{GI}$, which is (1).
\end{proof}
\begin{remark}\label{remzad}
The equivalence between (1) and (4) in Corollary~\ref{cor-GP-virtually-Gor} was proved for commutative noetherian rings of finite Krull dimension in \cite[Theorem 3.5]{ZAD}. The direct comparison of orthogonals above removes these restrictions. By Example~\ref{vG}(3), for each $0\leq n\leq\infty$ there are noetherian rings of Krull dimension $n$ that are not virtually Gorenstein; the theorem shows that $\mathcal{GP}\times\mathcal{GI}$ is not balanced for any of these rings.
\end{remark}

For another particular case of the classes $\mathcal{DP}$ and $\mathcal{DI}$, we obtain

\begin{corollary}\label{ding-balance-arbitrary}
Let $R$ be a ring. Then the following conditions are equivalent:
\begin{enumerate}
\item[(1)] $\mathcal{DP}\times\mathcal{DI}$ is a balanced pair in $\rmod R$.
\item[(2)] $\mathcal{DP}\times\mathcal{DI}$ is a special balanced pair in $\rmod R$.
\item[(3)] $\mathcal{DP}^{\perp_1}={}^{\perp_1}\mathcal{DI}$.
\item[(4)] $(\mathcal{DP},\mathcal{DP}^{\perp_1},\mathcal{DI})$ is a complete hereditary cotorsion triple in $\rmod R$.
\end{enumerate}
In this case $\mathcal W=\mathcal{DP}^{\perp_1}={}^{\perp_1}\mathcal{DI}$ is definable, $(\mathcal{DP},\mathcal W)$ is generated by a set of strongly countably presented modules, and $(\mathcal W,\mathcal{DI})$ is perfect.
\end{corollary}
\begin{proof}
$(1)\Rightarrow(4)$. The classes $\mathcal{DP}$ and $\mathcal{DI}$ are closed under direct summands, $\mathcal{DP}$ is resolving, and $\mathcal{DI}$ is coresolving. 
The defining acyclicity conditions give
$\mathcal P_0\subseteq\mathcal{DP}^{\perp_1}$
and $\mathcal I_0\subseteq{}^{\perp_1}\mathcal{DI}$. It follows from \cite[Theorem 44]{GI22} that $\mathcal{DI}$ is special preenveloping. Thus Theorem~\ref{hered}(4) applies and gives (4).

$(4)\Rightarrow(2)$ follows from Theorem~\ref{hered}, while $(2)\Rightarrow(1)$ and $(4)\Rightarrow(3)$ follow from the definitions.

$(3)\Rightarrow(4)$. Put $\mathcal W=\mathcal{DP}^{\perp_1}={}^{\perp_1}\mathcal{DI}$. It follows from \cite[Theorem 3.3]{CS} that $(\mathcal{DP},\mathcal W)$ is a hereditary cotorsion pair. By \cite[Theorem 44]{GI22}, $(\mathcal W,\mathcal{DI})$ is a complete hereditary cotorsion pair. It follows that $\mathcal W$ is thick. Moreover, $\mathcal W={}^{\perp_1}\mathcal{DI}$ is filtration closed by the Eklof Lemma, so it is closed under direct limits by \cite[Lemma 5.1]{SaS}. Now \cite[Theorem 3.5]{SaSt} implies that $(\mathcal{DP},\mathcal W)$ is complete, proving (4). The remaining assertions follow from Theorem~\ref{hered}.
\end{proof}

\begin{corollary}\label{balanced-GP-DP-PGF-GI-DI}
If $\mathcal{GP}\times\mathcal{GI}$ is balanced in $\rmod R$, then
$\mathcal{GP}=\mathcal{DP}=\mathcal{PGF}$ and $\mathcal{GI}=\mathcal{DI}$.
In particular, $\mathcal{DP}\times\mathcal{DI}$ is the same balanced pair, and every Gorenstein projective module is Gorenstein flat.
\end{corollary}
\begin{proof} The hypotheses of Proposition~\ref{balanced-hered}(1) were verified in the proof of Corollary~\ref{virtGor}. Applying Corollary~\ref{balanced-structural-gorenstein}, we obtain
$\mathcal{GP}\subseteq\mathcal A\subseteq\mathcal{PGF}
 \subseteq\mathcal{DP}\subseteq\mathcal{GP}$ and
 $\mathcal{GI}\subseteq\mathcal B\subseteq\mathcal{DI}\subseteq\mathcal{GI}$.
These inclusions prove the asserted equalities. The remaining assertion follows from the 
fact that every projectively coresolved Gorenstein flat module is Gorenstein flat.
\end{proof}

\begin{remark}
Gillespie proved that $\mathcal{GP}=\mathcal{DP}$ and $\mathcal{GI}=\mathcal{DI}$ over Ding--Chen rings, c.f. \cite[Theorem~1.1]{GillespieDing}. Corollaries~\ref{cor-GP-virtually-Gor} and \ref{balanced-GP-DP-PGF-GI-DI} extend these equalities to all right virtually Gorenstein rings.
Ding--Chen rings form a subclass of the right virtually Gorenstein
rings, and Example~\ref{vG}(2) shows that this inclusion is strict.
\end{remark}

%{\color{red}
%\begin{remark}
%Gillespie proved that $\mathcal{GP}=\mathcal{DP}$ and
%$\mathcal{GI}=\mathcal{DI}$ over Ding--Chen rings, c.f. 
%\cite[Theorem~1.1]{GillespieDing}. Since Ding--Chen rings  form a proper
%subclass of the right virtually Gorenstein rings by 
%Example~\ref{vG}(2), Corollaries~\ref{cor-GP-virtually-Gor} and
%\ref{balanced-GP-DP-PGF-GI-DI} give the same equalities
%in a more general setting.
%\end{remark}}

Next, we employ the balance in order to obtain further closure properties:

\begin{lemma}\label{sums} Let $R$ be a ring and $\mathcal F \times \mathcal G$ be a balanced pair in $\rmod R$.
\begin{enumerate}
\item[(1)] Assume that the class $\mathcal F ^{\perp_1}$ contains all direct sums of injective modules (e.g., assume that $R$ is right noetherian). Then $\mathcal F ^{\perp_1}$ is closed under direct sums.
\item[(2)] Assume that the class $\mathcal F^{\perp_{\infty}}$ contains all direct sums of injective modules (e.g., assume that $R$ is right noetherian). Then $\mathcal F ^{\perp_{\infty}}$ is closed under direct sums.
\item[(3)] Assume that the class $\mathcal F ^{\perp_1}$ is coresolving. Moreover, assume that $\mathcal F ^{\perp_1}$ contains all $\mathcal I _0$-filtered modules (e.g., assume that $R$ is right noetherian). Then $\mathcal F ^{\perp_1}$ is filtration closed.
\end{enumerate}
\end{lemma}
\begin{proof} (1) The result follows from \cite[Lemma 5.1(1)]{EPZ}.

(2) Let $\{Y_{i}\}_{i\in{I}}$ be any family of modules in $\mathcal F^{\perp_{\infty}}$. By (1), $\textrm{Ext}_{R}^{1}(F,\oplus_{i\in{I}}Y_{i})=0$. For each $i \in I$, consider an exact sequence $0\to Y_{i}\to E_{i}\to L_{i}\to 0$ in $\rmod R$ with $E_{i}$ injective. By dimension shifting, $\{L_{i}\}_{i\in{I}}$ is also a family of modules in $\mathcal F^{\perp_{\infty}}$. So again by (1), $\textrm{Ext}_{R}^{1}(F,\oplus_{i\in{I}}L_{i}) = 0$ for each $F\in{\mathcal{F}}$. Moreover, since the sequence $0\to \oplus_{i\in{I}} Y_{i}\to \oplus_{i\in{I}} E_{i} \to \oplus_{i\in{I}} L_{i}\to 0$ is exact and the module $\oplus_{i\in{I}} E_{i}$ belongs to $\mathcal F^{\perp_{\infty}}$ by assumption, we obtain $\textrm{Ext}_{R}^{2}(F,\oplus_{i\in{I}}Y_{i})\cong \textrm{Ext}_{R}^{1}(F,\oplus_{i\in{I}}L_{i}) = 0$ for each $F\in{\mathcal{F}}$.

Applying the reasoning above to the family $\{L_{i}\}_{i\in{I}}$, we get $0 = \textrm{Ext}_{R}^{2}(F,\oplus_{i\in{I}}L_{i})$, and hence $\textrm{Ext}_{R}^{3}(F,\oplus_{i\in{I}}Y_{i}) = 0$. Proceeding by induction, we conclude that for all $k\geq 1$ and $F\in{\mathcal{F}}$, $\textrm{Ext}_{R}^{k}(F,\oplus_{i\in{I}}Y_{i})=0$.

(3) Let $\mathcal M = (M_\alpha \mid \alpha \leq \sigma)$ be an $\mathcal F ^{\perp_1}$-filtration of a module $M$. We will prove that $M \in \mathcal F ^{\perp_1}$. By (the dual of) \cite[Lemma 7.1]{GT}, we can expand $\mathcal M$ into a well-ordered continuous direct system of short exact sequences
$$C(\mathcal M)  \, =  \, ( 0 \to M_\alpha \to I_\alpha \to K_\alpha \to 0 \mid \alpha \leq \sigma ),$$
i.e., for each $\alpha < \sigma$, there is a commutative diagram
$$\begin{CD}
@.     0@.        0@.           0                  @.\\
@.     @VVV       @VVV   @VVV     @.\\
0@>>>  {M_\alpha}@>>> {I_\alpha}@>>>  {K_\alpha}@>>>  0\\
@.     @VVV  @VVV   @VVV     @.\\
0@>>>  {M_{\alpha +1}}@>>> {I_{\alpha +1}}@>>>  {K_{\alpha +1}}@>>>  0\\
@.     @VVV       @VVV  @VVV  @.\\
0@>>>  {\bar{M}_\alpha} @>>>   {\bar{I}_{\alpha}}@>>>    {\bar{K}_\alpha}@>>> 0\\
@.     @VVV       @VVV   @VVV    @.\\
@.     0@.        0@.    \;0.                          @.
\end{CD}$$
with exact rows and columns, and with $\bar{I}_\alpha$ injective. By our assumption on $\mathcal F ^{\perp_1}$, we have $I_\alpha \in \mathcal F ^{\perp_1}$. Since $\mathcal F ^{\perp_1}$ is coresolving, all modules in the diagram above belong to $\mathcal F ^{\perp_1}$. Because $\mathcal F \times \mathcal G$ is balanced, we infer from Lemma \ref{enochs} that for each $G \in \mathcal G$, also the diagram
\begingroup\ifshowchanges\fi\small\setlength{\minCDarrowwidth}{1.5em}
$$\begin{CD}
@.     0@.        0@.           0                  @.\\
@.     @VVV       @VVV   @VVV     @.\\
0@>>>  {\Hom R{\bar{K}_\alpha}G}@>>> {\Hom R{\bar{I}_\alpha}G}@>>>  {\Hom R{\bar{M}_\alpha}G}@>>>  0\\
@.     @VVV  @VVV   @VVV     @.\\
0@>>>  {\Hom R{K_{\alpha +1}}G}@>>> {\Hom R{I_{\alpha +1}}G}@>>>  {\Hom R{M_{\alpha +1}}G}@>>>  0\\
@.     @VVV       @VVV  @VVV  @.\\
0@>>>  {\Hom R{K_{\alpha}}G}@>>> {\Hom R{I_{\alpha}}G}@>>>  {\Hom R{M_{\alpha}}G}@>>>  0\\
@.     @VVV       @VVV   @VVV    @.\\
@.     0@.        0@.    \;0.                          @.
\end{CD}$$
\endgroup
is commutative and has exact rows and columns. Since the functor $\varprojlim$ is exact at well-ordered continuous inverse systems of short exact sequences \cite[Lemma 6.36]{GT}, also the sequence $$0 \to \varprojlim \Hom R{K_\alpha}G \to \varprojlim \Hom R{I_\alpha}G \to \varprojlim \Hom R{M_\alpha}G \to 0$$
is exact. Moreover, there is a commutative diagram
\begingroup\ifshowchanges\fi\small\setlength{\minCDarrowwidth}{1.5em}
$$\begin{CD}
0@>>>  {\Hom R{\varinjlim K_\alpha}G}@>>> {\Hom R{\varinjlim I_\alpha}G}@>>>  {\Hom RMG}@>>>  0\\
@.     @VVV  @VVV   @VVV     @.\\
0@>>>  {\varprojlim \Hom R{K_\alpha}G}@>>> {\varprojlim \Hom R{I_\alpha}G}@>>>  {\varprojlim \Hom R{M_\alpha}G}@>>>  0\\
\end{CD}$$
\endgroup
where the vertical maps are canonical isomorphisms (see e.g.\ \cite[\S6.3]{GT}). It follows that the sequence
$0 \to M \to \varinjlim_{\alpha < \sigma} I_\alpha \to \varinjlim_{\alpha < \sigma} K_\alpha \to 0$ is $\Hom R{-}G$ exact for each $G \in \mathcal G$. By Lemma \ref{enochs}, it is also $\Hom RF{-}$ exact for each $F \in \mathcal F$. Since $\varinjlim _{\alpha < \sigma} I_\alpha$ is $\mathcal I_0$-filtered, we have $\varinjlim _{\alpha < \sigma} I_\alpha \in \mathcal F ^{\perp_1}$, whence $M \in \mathcal F ^{\perp_1}$.
\end{proof}

\begin{corollary}\label{cor:GP-direct-limits} Let $R$ be a right noetherian ring and $\mathcal F \times \mathcal G$ be a balanced pair in $\rmod R$.
 \begin{enumerate}
\item[(1)] If $\mathcal F ^{\perp_1}$ is thick, then $\mathcal F ^{\perp_1}$ is closed under direct limits.
\item[(2)] Assume that  $(^{\perp_1}(\mathcal{F}^{\perp_1}), \mathcal{F}^{\perp_1})$ is a perfect hereditary cotorsion pair and the class ${\mathcal{F}^{\perp_1}}$ is closed under kernels of epimorphisms. Then $^{\perp_1}(\mathcal{F}^{\perp_1})$ is closed under direct limits.
\end{enumerate}
\end{corollary}

\begin{proof} (1) This follows by Lemma \ref{sums}(3) and \cite[Lemma 5.1]{SaS}.

(2) By the assumptions, ${\mathcal{F}^{\perp_1}}$ is thick, so part (1) yields that $\mathcal F ^{\perp_1}$ is closed under direct limits. By \cite[Theorem 5.2]{AST}, also $^{\perp_1}(\mathcal{F}^{\perp_1})$ is closed under direct limits.
\end{proof}

The following result provides a perfect Gorenstein counterpart to a theorem of Enochs \cite{Enochs15}, which establishes that the class of flat left $R$-modules cannot, in general, form the left part of a balanced pair.

\begin{proposition}\label{prop:GF-balanced-pair}
Let $R$ be a right noetherian and right virtually Gorenstein ring. Then $\mathcal{GF} \times \mathcal{L}$ is a balanced pair for some class $\mathcal{L}$ in $\rmod R$ if and only if $R$ is right perfect.
\end{proposition}
\begin{proof} For the ``if'' part, we assume that $R$ is right perfect. Then $\mathcal{GF}$ coincides with the class of projectively coresolved Gorenstein flat right $R$-modules. Since $R$ is a right virtually Gorenstein ring, it follows from \cite[Proposition 3.16]{WE25} that $\mathcal{GP}$ is exactly the class of projectively coresolved Gorenstein flat right $R$-modules, and therefore we have $\mathcal{GP}=\mathcal{GF}$. Now, the statement follows by Theorem \ref{cor-GP-virtually-Gor}.

For the ``only if'' part, we assume that $\mathcal{GF} \times \mathcal{L}$ is a balanced pair for some class $\mathcal{L}$ in $\rmod R$. Since $R$ is right noetherian, it follows from Lemma \ref{sums}(1) that $\mathcal{GF}^{\perp_1}$ is closed under direct sums. The class $\mathcal{GF}$ is also closed under direct sums. Now, note that $\mathcal{GF}\cap\mathcal{GF}^{\perp_1}$ is exactly the class of all flat and cotorsion modules by \cite[Corollary 4.12]{SaS}. Hence, by a result of Guil Asensio and Herzog \cite[Theorem 19]{GH}, $R$ is right perfect.
\end{proof}

Dually to Lemma \ref{sums}, we obtain

\begin{lemma}\label{prods} Let $R$ be a ring and $\mathcal F \times \mathcal G$ be a balanced pair in $\rmod R$.
\begin{enumerate}
\item[(1)] Assume that the class $^{\perp_1}\mathcal G$ contains all direct products of projective modules. Then $^{\perp_1}\mathcal G$ is closed under direct products.
\item[(2)] Assume that the class $^{\perp_{\infty}}\mathcal G$ contains all direct products of projective modules. Then $^{\perp_{\infty}}\mathcal G$ is closed under direct products.
\item[(3)] Assume that the class $^{\perp_1}\mathcal G$ is resolving. Moreover, assume that $^{\perp_1}\mathcal G$ contains all $\mathcal P _0$-cofiltered modules. Then $^{\perp_1}\mathcal G$ is cofiltration closed.
\end{enumerate}
\end{lemma}

By the Eklof Lemma \cite[Lemma 6.2]{GT}, the class $^{\perp_1} \mathcal C$ is filtration closed for each class of modules $\mathcal C$. This is not necessarily the case of the class $\mathcal C ^{\perp _1}$ (which need not even be closed under direct sums). In this sense, Lemma \ref{sums} gives a considerable restriction on the class $\mathcal F$.

Note however that for each tilting module $T$, the tilting class $T^{\perp_\infty}$ is filtration closed (see e.g.\ \cite[Lemma 13.27]{GT}). For each $n \geq 0$, let $\mathcal P_n$ and $\mathcal I_n$ denote the classes of all modules of projective and injective dimension $\leq n$, respectively. Then $(\mathcal P_n,\mathcal P_n ^\perp)$ and $(^\perp \mathcal I_n,\mathcal I_n)$ are complete hereditary cotorsion pairs (see e.g.\ \cite[Theorems 8.7 and 8.10]{GT}), whence Lemmas \ref{sums} and \ref{prods} make it possible to rephrase the characterizations of tilting and cotilting cotorsion pairs provided by Lemmas \ref{1tilt} and \ref{1cotilt} in the setting of balanced pairs as follows:

\begin{corollary}\label{tilt} Let $R$ be a ring, $n \geq 0$, and $\mathcal F \times \mathcal G$ be a balanced pair in $\rmod R$.

Then $(^{\perp_1}(\mathcal{F}^{\perp_{\infty}}), \mathcal{F}^{\perp_{\infty}})$ is an $n$-tilting cotorsion pair, if and only if $\mathcal{F} \subseteq \mathcal P _n$ and $\mathcal F ^{\perp_{\infty}}$ contains all direct sums of injective modules.
\end{corollary}

\begin{example} Let $R$ be a right coherent ring and $n \geq 0$. Consider the balanced pair $\mathcal{PP}\times \mathcal{PI}$ in $\rmod R$. We have $\mathcal{PP}^{\perp_{\infty}}=\mathcal{FPI}$ because $R$ is right coherent, so Corollary \ref{tilt} yields that the cotorsion pair $\mathfrak C = (^{\perp_1}\mathcal{FPI}, \mathcal{FPI})$ is $n$-tilting, if and only (a) the class $\mathcal{FPI}$ contains all direct sums of injective modules, and (b) $\mathcal{PP} \subseteq \mathcal P _n$. However, condition (a) is satisfied (even for any ring $R$), since the class $\mathcal{FPI}$ is closed under direct sums. Condition (b) is equivalent to all finitely presented modules being of weak dimension $\leq n$ by \cite[Lemma 2.16]{GT}, and hence to $R$ being of weak global dimension $\leq n$.
\end{example}

\begin{corollary}\label{cotilt} Let $R$ be a ring, $n \geq 0$, and $\mathcal F \times \mathcal G$ be a balanced pair in $\rmod R$.

Then $(^{\perp_{\infty}}\mathcal G, (^{\perp_{\infty}}\mathcal G)^{\perp_1})$ is an $n$-cotilting cotorsion pair, if and only if $\mathcal{G} \subseteq \mathcal I _n$  and $^{\perp_{\infty}}\mathcal G$ contains all direct products of projective modules.
\end{corollary}

\begin{example} Let $R$ be any ring and $n \geq 0$. Consider the balanced pair $\mathcal{PP}\times \mathcal{PI}$ in $\rmod R$. Since $^{\perp_{\infty}}\mathcal{PI}=\mathcal{F}_0$, Corollary \ref{cotilt} yields that the cotorsion pair $\mathfrak C = (\mathcal{F}_0, \mathcal{F}_0^{\perp_1})$ is $n$-cotilting, if and only (a) the class $\mathcal F_0$ contains all direct products of projective modules, and (b) $\mathcal{PI} \subseteq \mathcal I _n$. By a classic theorem of Chase, (a) is equivalent to $R$ being left coherent, while (b) is equivalent to $R$ having weak global dimension $\leq n$. Indeed, since $\mathfrak C$ is cogenerated by the class of all dual modules, (b) is equivalent to all dual modules having injective dimension $\leq n$, and hence to all (left $R$-) modules having weak dimension $\leq n$.
\end{example}

From Lemma \ref{sums}, we also get

\begin{corollary}\label{tilt'''} Let $R$ be a right noetherian ring with finite global dimension and $\mathcal F \times \mathcal G$ be a balanced pair in $\rmod R$. Then $(^{\perp_1}(\mathcal{F}^{\perp_{\infty}}), \mathcal{F}^{\perp_{\infty}})$ is a tilting cotorsion pair.
\end{corollary}

\begin{remark} We can construct tilting cotorsion pairs from balanced pairs using Corollaries \ref{tilt} and \ref{tilt'''}. But it is difficult to find a balanced pair $\mathcal F \times \mathcal G$ such that $(\mathcal{F}, \mathcal{F}^{\perp_1})$ is a cotorsion pair. For instance, if we assume that $\mathcal F$ is the class of pure-projective $R$-modules and $\mathcal G$ is the class of
pure-injective $R$-modules, then $(\mathcal{F}, \mathcal{F}^{\perp_1})$ is not a cotorsion pair in general, but $(^{\perp_1}(\mathcal{F}^{\perp_1}), \mathcal{F}^{\perp_1})= (^{\perp_1}\mathcal{FPI}, \mathcal{FPI})$ is a cotorsion pair.
\end{remark}

Next, we turn again to cotorsion triples. First, we observe that a cotorsion triple $(\mathcal A,\mathcal B,\mathcal C)$ whose first component $(\mathcal A,\mathcal B)$ is tilting must be trivial:

\begin{proposition}\label{trivial} Let $R$ be a ring and $t = (\mathcal A,\mathcal B,\mathcal C)$ be a cotorsion triple such that the cotorsion pair $\mathfrak C = (\mathcal A,\mathcal B)$ is tilting. Then $t$ is trivial (i.e., $\mathcal A = \mathcal P_0$).
\end{proposition}
\begin{proof} Assume $\mathfrak C$ is tilting. Then $\mathcal A \subseteq \mathcal P$. Since $\mathcal B \supseteq \mathcal P _0$ and $\mathcal B$ is coresolving, $\mathcal B \supseteq \mathcal P$. Thus $\mathcal A \subseteq \mathcal B$.

Let $T$ be a tilting module such that $\mathfrak C$ is induced by $T$. By \cite[Lemma 13.10(c)]{GT}, Add$(T) = \mathcal A \cap \mathcal B = \mathcal A$. Moreover, by \cite[Proposition 13.13(b)]{GT}, the class $\mathcal B$ coincides with the class of all modules possessing a (possibly infinite) Add$(T)$-resolution. Since $\mathcal P _0 \subseteq \mathcal A$, we infer that $\mathcal B = \rmod R$, and $\mathcal A = \mathcal P _0$.
\end{proof}

By Example \ref{vG}(1), for any Iwanaga-Gorenstein ring $R$, the complete hereditary cotorsion pair $(\mathcal P,\mathcal{GI})$ is tilting, and it forms the second component in the cotorsion triple $(\mathcal{GP},\mathcal P, \mathcal{GI})$. In fact, this is the only non-trivial possibility:

\begin{proposition}\label{almosttrivial} Let $R$ be a ring and $t = (\mathcal A,\mathcal B,\mathcal C)$ be a cotorsion triple such that the cotorsion pair $\mathfrak C = (\mathcal B,\mathcal C)$ is tilting. Then $\mathcal B \supseteq \mathcal I$ and $\mathcal C \supseteq \mathcal{GI}$. Moreover
\begin{enumerate}
\item[(1)] If $R$ has finite right global dimension, then $t = (\mathcal P _0,\rmod R,\mathcal I _0)$, and $R$ is right noetherian.
\item[(2)] If $R$ is Iwanaga-Gorenstein, then $t = (\mathcal{GP},\mathcal P,\mathcal{GI})$.
\item[(3)] Let $R$ be the virtually Gorenstein ring from Example \ref{vG}(2). Then there exists no cotorsion triple $t = (\mathcal A,\mathcal B,\mathcal C)$ such that the cotorsion pair $\mathfrak C = (\mathcal B,\mathcal C)$ is tilting.
\end{enumerate}
\end{proposition}
\begin{proof} Assume $\mathfrak C$ is tilting. Then $\mathcal B \subseteq \mathcal P$. Since $\mathcal B \supseteq \mathcal I _0$ and $\mathcal B$ is resolving, $\mathcal B \supseteq \mathcal I$.

Let $T$ be a tilting module such that $\mathfrak C$ is induced by $T$. By \cite[Lemma 13.10(c)]{GT}, Add$(T) = \mathcal B \cap \mathcal C \supseteq \mathcal I _0$. So by \cite[Proposition 13.13(a)]{GT}, the class $\mathcal C$ contains the class of all modules possessing a complete injective resolution. In particular, $\mathcal C \supseteq \mathcal{GI}$.

(1) If $R$ has finite right global dimension, then $\mathcal B = \mathcal I = \rmod R$ and $\mathcal C = \mathcal I _0$. By Lemma
\ref{1tilt}, the class $\mathcal C$ is closed under direct sums, that is, $R$ is right noetherian.

(2) If $R$ is Iwanaga-Gorenstein, then $(\mathcal I,\mathcal{GI})$ is a cotorsion pair, so the two inclusions above are actually equalities: $\mathcal B = \mathcal I (= \mathcal P)$, and $\mathcal C = \mathcal{GI}$.

(3) By \cite[Example 2.15]{DLW}, $R$ is a commutative local artinian ring with the maximal ideal $m$ such that $m^2 = 0$. In particular, $\Spec R = \Ass {R}{R} = \{ m \}$, so \cite[Theorem 16.4]{GT} yields that all tilting modules are projective. However, $(^\perp \mathcal P_0,\mathcal P_0)$ is not a cotorsion pair in $\rmod R$ (otherwise $\mathcal I_0 \subseteq \mathcal P _0$, so $R$ is quasi-Frobenius, whence $R \in \mathcal I _0$, a contradiction).
\end{proof}

\section{Cotorsion torsion triples and their duals}\label{Section4}

The Hom- and Ext- orthogonal pairs, i.e., the torsion and the cotorsion pairs, in an abelian category can be viewed as sort of orthogonal decompositions of the category. In \cite{BBOS}, this fact was employed in developing the representation theory of certain rectangular grids occurring in persistent homology theory. The central notion in \cite{BBOS} combined the two decompositions into one, namely the notion of a \emph{cotorsion torsion triple} in an abelian category $\mathcal A$ with enough projectives. Recently, this notion was further studied  by Buan and Zhou in \cite{BZ} in the setting of finite dimensional algebras.

One of the main results of \cite{BBOS} (see also \cite{B}) was the classification of cotorsion torsion triples in $\mathcal A$ in terms of tilting subcategories of $\mathcal A$. Our goal here is to generalize and extend, in the particular setting of module categories, the results from \cite{BBOS} and \cite{B} and reveal thus possible new approaches to cotorsion torsion triples in categories of modules. Our main tool is infinite dimensional tilting theory of modules, cf.\ \cite[Part III]{GT}.

In Theorem \ref{ctt}, we show that $1$-resolving subcategories $\mathcal S$ of the category $\rfmod R$ of all finitely presented modules classify cotorsion torsion triples $c = (\mathcal C,\mathcal T,\mathcal F)$ in $\rmod R$ via the assignment $\mathcal S \mapsto c_{\mathcal S}$ where $c_{\mathcal S}$ is the triple with the middle term $\mathcal T = \mathcal S ^{\perp_1}$. The classic Ext-Tor duality then yields a monic map from the cotorsion torsion triples in $\rmod R$ to the dual setting of torsion cotorsion triples in $\lmod R$ as follows: the cotorsion torsion triple $c_{\mathcal S}$ in $\rmod R$ is mapped to the torsion cotorsion triple $d_S = (\mathcal T,\mathcal F,\mathcal D)$ in $\lmod R$ with the middle term $\mathcal F = \mathcal S ^{\intercal}$. This map is bijective when $R$ is left noetherian, providing an explicit expression for the formal duality between the two settings (Corollary~\ref{explicit}).

\medskip
We start by recalling the relevant definitions.

\begin{definition}\label{general} Let $R$ be a ring.
A triple of classes of modules $(\mathcal C, \mathcal T, \mathcal F)$ is a \emph{cotorsion torsion triple} in case $(\mathcal C, \mathcal T)$ is a cotorsion pair, and $(\mathcal T, \mathcal F)$ a torsion pair in $\rmod R$.

Dually, a triple of classes of left $R$-modules $(\mathcal T, \mathcal F, \mathcal D)$ is a \emph{torsion cotorsion triple} in case $(\mathcal T, \mathcal F)$ is a torsion pair, and $(\mathcal F, \mathcal D)$ a cotorsion pair in $\lmod R$.
\end{definition}

Definition \ref{general} is more general than \cite[Definitions 2.6 and 2.9]{BBOS} (or the corresponding definitions in \cite{B}), where all the cotorsion pairs involved are assumed to be complete, cf.\ \cite[Remark 2.7]{BBOS}. We will see that in the setting of module categories, the structure results for cotorsion torsion triples and torsion cotorsion triples obtained in \cite[\S 2]{BBOS} do not require the a priori assumption of completeness: completeness follows in this case from the other properties assumed in our Definition \ref{general}.

In contrast, our next definition is exactly the same as \cite[Definition 2.16 and \S2.3]{BBOS}:

\begin{definition}\label{wtilt} Let $R$ be a ring.

\begin{enumerate}
\item[(1)] A full subcategory $\mathbb T$ of $\rmod R$ which is closed under finite direct sums and direct summands is a \emph{weak tilting} subcategory in case $\mathbb T$ satisfies the following three conditions:

(i) $\Ext {1}{R}{T_1}{T_2} = 0$ for all $T_1, T_2 \in \mathbb T$,

(ii) $\mathbb T$ consists of modules of projective dimension $\leq 1$, and

(iii) for each projective module $P$ there exists a short exact sequence $0 \to P \to T_0 \to T_1 \to 0$ where $T_0, T_1 \in \mathbb T$.

A weak tilting subcategory $\mathbb T$ is \emph{tilting} in case $\mathbb T$ is a precovering class in $\rmod R$.

\item[(2)] A full subcategory $\mathbb C$ of $\lmod R$ closed under finite direct products and direct summands is a \emph{weak cotilting} subcategory in case it satisfies the following three conditions

(i) $\Ext {1}{R}{C_1}{C_2} = 0$ for all $C_1, C_2 \in \mathbb C$,

(ii) $\mathbb C$ consists of modules of injective dimension $\leq 1$, and

(iii) for each injective module $I$ there exists a short exact sequence $0 \to C_0 \to C_1 \to I \to 0$ where $C_0, C_1 \in \mathbb C$.

A weak cotilting subcategory $\mathbb C$ is \emph{cotilting} in case $\mathbb C$ is a preenveloping class in $\lmod R$.
\end{enumerate}
\end{definition}

Notice that since all precovering/preenveloping classes are closed under (arbitrary) direct sums/products by Lemma \ref{easy}, this in particular holds for the tilting/cotilting subcategories from Definition \ref{wtilt}. In Corollary \ref{weakclosed}, we will see that for module categories, the converse holds in the weak setting, i.e., closure under (arbitrary) direct sums/products already implies the precovering/preenveloping property.

\cite[Example 2.21]{BBOS} provides an example of a weak tilting subcategory which is not tilting in the abelian category of all finitely presented representations of the totally ordered set of all positive real numbers. Here, we provide an example of this phenomenon in $\rmod R$:

\begin{example}\label{exreg} (1) Let $R$ be any right hereditary ring which is not right noetherian (e.g., let $K$ be a field and $R$ be the ring of all eventually constant sequences of elements of $K$).

Let $\mathbb T$ denote the category of all injective modules. Since $R$ is right hereditary, $\mathbb T$ is a weak tilting subcategory of $\rmod R$. However, since $R$ is not right noetherian, $\mathbb T$ is not closed under direct sums, so $\mathbb T$ is not a tilting subcategory of $\rmod R$.

(2) There is a particular instance of (1) where the class $\mathbb T$ of all injective modules even does not contain \emph{any} infinite direct sums of non-zero modules:

Let $R$ be a simple countable von Neumann regular ring such that $R$ is not completely reducible. Such rings $R$ can be constructed as subrings of the simple rings of the form $E/M$ where $E$ is the endomorphism ring of an infinite dimensional linear space $L$, and $M$ is the maximal two-sided ideal of $L$ (consisting of the endomorphisms of rank less than the dimension of $L$).

Since $R$ is countable, $R$ is hereditary. Since $R$ is simple, $\Ext {1}{R}{R/I}{N} \neq 0$ whenever $N$ is an infinite direct sum of non-zero modules, and $I$ is an infinitely (countably) generated right ideal in $R$. In particular, no infinite direct sum of non-zero modules is injective.
\end{example}

\medskip

As we do not assume completeness of cotorsion pairs a priori, our proofs are quite different from those given in \cite{BBOS} and \cite{B}. However, by employing tools of infinite dimensional tilting theory of modules \cite[Part III]{GT}, we can proceed directly to the main characterizations:

\begin{theorem}\label{ctt} Let $R$ be a ring. Then there is a bijective correspondence between
\begin{enumerate}
\item[(1)] cotorsion torsion triples $(\mathcal C, \mathcal T, \mathcal F)$ in $\rmod R$,

\item[(2)] $1$-tilting torsion classes $\mathcal B$ in $\rmod R$,

\item[(3)] tilting subcategories $\mathbb T$ of $\rmod R$,

\item[(4)] weak tilting subcategories $\mathbb T$ of $\rmod R$ closed under direct sums, and

\item[(5)] $1$-resolving subcategories $\mathcal S$ of $\rfmod R$.
\end{enumerate}
\end{theorem}

\begin{proof}
(1) $\leftrightarrow$ (2). Let $(\mathcal C, \mathcal T, \mathcal F)$ be a cotorsion torsion triple in $\rmod R$. Put $\mathcal B = \mathcal T$. Since the torsion class $\mathcal B$ is closed under homomorphic images and contains all injective modules, it also contains all first cosyzygy modules. Hence the class $\mathcal C = {^{\perp_{1}} \mathcal B}$ consists of modules of projective dimension $\leq 1$. As $\mathcal B$ is closed under direct sums, $(\mathcal C,\mathcal B)$ is a $1$-tilting cotorsion pair by Lemma \ref{1tilt}. That is, $\mathcal B$ is a $1$-tilting torsion class in $\rmod R$.

Conversely, each $1$-tilting torsion class $\mathcal B$ in $\rmod R$ induces the $1$-tilting cotorsion pair $(\mathcal A, \mathcal B)$ where $\mathcal A = {^{\perp_{1}} \mathcal B}$, and a torsion pair $(\mathcal B, \mathcal F)$ where $\mathcal F = \mbox{Ker} \Hom{R}{\mathcal T}{-}$, see e.g.\ \cite[p.337]{GT}.

So a bijective correspondence between (1) and (2) is given by $(\mathcal C, \mathcal T, \mathcal F) \mapsto \mathcal T$.

(2) $\rightarrow$ (3). Let $T$ be a $1$-tilting module such that $\mathcal B = T^{\perp_{1}}$. We put $\mathbb T = \Add {T}$, so $\mathbb T = \mathcal B \cap {}^{\perp_{1}} \mathcal B$ by \cite[Lemma 13.10(c)]{GT}. Clearly, $\mathbb T$ satisfies conditions (i)-(iii) \ from Definition \ref{wtilt}(1). In order to prove that $\mathbb T$ is a tilting subcategory, it remains to show that $\mathbb T$ is precovering in $\rmod R$.

For each $M \in\rmod R$ let $t(M)$ denote the $\mathcal B$-torsion part of $M$, and let $\nu : t(M) \hookrightarrow M$ be the inclusion. Since $\mathcal B$ is a $1$-tilting class, there is a complete cotorsion pair $(\mathcal A,\mathcal B)$, so we have a short exact sequence $0 \to B \to A \overset{\pi}\to t(M) \to 0$ where $B \in \mathcal B$, and $A \in \mathcal A \cap \mathcal B = \mathbb T$ by \cite[Lemma 13.10(c)]{GT}.

Let $N \in \mathbb T$ and $f \in \Hom {R}{N}{M}$. Then im$(f) \subseteq t(M)$. Since $\Ext {1}{R}{N}{B} = 0$, there exists $g \in \Hom {R}{N}{A}$ such that $\nu \pi g = f$. This proves that $\nu \pi$ is a $\mathbb T$-precover of $M$.

The correspondence from (2) to (3) is given by $\alpha : \mathcal B \mapsto \mathcal B \cap {^{\perp_{1}} \mathcal B}$, and we just take the identity map for the correspondence from (3) to (4).

(4) $\rightarrow$ (2). Let $\mathcal B = \mathbb T ^{\perp_{1}}$ and $\mathcal A = {^{\perp_{1}} \mathcal B}$. Then $(\mathcal A,\mathcal B)$ is a cotorsion pair in $\rmod R$. Moreover, by parts (i) and (ii)\ of Definition \ref{wtilt}(1), $\mathbb T \subseteq \mathcal A \cap \mathcal B$ and the class $\mathcal A$ consists of modules of projective dimension $\leq 1$. So the class $\mathcal B$ is closed under homomorphic images. In particular, any homomorphic image of a module from $\mathbb T$ belongs to $\mathcal B$.

Conversely, let $B \in \mathcal B$ and let $f \in \Hom {R}{P}{B}$ be an epimorphism with $P$ a projective module. By part (iii)\ of Definition \ref{wtilt}(1), we have a short exact sequence $0 \to P \overset{\nu}\to T_0 \to T_1 \to 0$ with $T_0, T_1 \in \mathbb T$. Taking the pushout of $f$ and $\nu$, we obtain a short exact sequence $0 \to B \to X \to T_1 \to 0$ where $X$ is a homomorphic image of $T_0$. Since $T_1 \in \mathcal A$ and $B \in \mathcal B$, the latter sequence splits, proving that $B$ is a homomorphic image of $T_0$.

Thus $\mathcal B$ is the class of all homomorphic images of modules from $\mathbb T$, hence $\mathcal B$ is closed under direct sums. By Lemma \ref{1tilt}, $(\mathcal A,\mathcal B)$ is a $1$-tilting cotorsion pair, and $\mathcal B$ is a $1$-tilting class.

So the correspondence from (4) to (2) is given by $\beta : \mathbb T \mapsto \mathbb T ^{\perp_{1}}$.

\vspace{2mm}
Notice that for a 1-tilting class $\mathcal B$, we have $\beta \alpha (\mathcal B) = (\mathcal B \cap {^{\perp_{1}} \mathcal B)}^{\perp_{1}}$. If $T$ is a 1-tilting module such that $T^{\perp_{1}} = \mathcal B$, then $\alpha (\mathcal B) = \Add {T}$, and $\beta (\Add {T}) = T^{\perp_{1}} = \mathcal B$. Hence $\beta \alpha = \hbox{id}$.

Conversely, let $\mathbb T$ be a weak tilting subcategory closed under direct sums. Then $\beta (\mathbb T) = \mathbb T ^{\perp_{1}}$ is a $1$-tilting class by the above. Let $T$ be a 1-tilting module such that $T^{\perp_{1}} = \mathbb T ^{\perp_{1}}$. Then $\Add {T} = \mathbb T ^{\perp_{1}} \cap {}^{\perp_{1}} (\mathbb T ^{\perp_{1}}) \supseteq \mathbb T$ by part (i)\ of Definition \ref{wtilt}(1).

By part (iii)\ of Definition \ref{wtilt}(1), there is a short exact sequence $0 \to R \to T_0 \to T_1 \to 0$ with $T_0, T_1 \in \mathbb T$. Let $T^\prime = T_0 \oplus T_1 \in \mathbb T$. By part (ii)\ of Definition \ref{wtilt}(1), $T^\prime$ has projective dimension $\leq 1$. Moreover, the closure of $\mathbb T$ under direct sums and part (i)\ of Definition \ref{wtilt}(1) give $\Ext {1}{R}{T^\prime}{(T^\prime)^{(I)}} = 0$ for any set $I$. Thus $T^\prime$ is also a $1$-tilting module.

Since $T^\prime \in \mathbb T$, we infer that $T^\prime \in \Add {T}$, whence $\mathbb T \supseteq \Add {T^\prime} = \Add {T} \supseteq \mathbb T$ by the above and by \cite[Lemma 13.16]{GT}. If follows that $\alpha \beta (\mathbb T) = \mathbb T ^{\perp_{1}} \cap {}^{\perp_{1}} (\mathbb T ^{\perp_{1}}) = \Add {T} = \mathbb T$, and $\alpha \beta = \mbox{id}$.

(2) $\leftrightarrow$ (5). By \cite[Corollary 14.7]{GT}, mutually inverse correspondences are given by the assignments $\mathcal B \mapsto \mathcal S$ where $\mathcal S$ is the class of all finitely presented modules in $^{\perp_{1}} \mathcal B$, and $\mathcal S \mapsto \mathcal S ^{\perp_{1}}$.
\end{proof}

In the dual setting, we have

\begin{theorem}\label{tct} Let $R$ be a ring. Then there is a bijective correspondence between
\begin{enumerate}
\item[(1)] torsion cotorsion triples $(\mathcal T, \mathcal F, \mathcal D)$ in $\lmod R$,

\item[(2)] $1$-cotilting torsion-free classes $\mathcal C$ in $\lmod R$,

\item[(3)] cotilting subcategories $\mathbb C$ of $\lmod R$, and

\item[(4)] weak cotilting subcategories $\mathbb C$ of $\lmod R$ closed under direct products.

\noindent In the case when $R$ is left noetherian, the bijective correspondence extends to

\item[(5)] $1$-resolving subcategories $\mathcal S$ of $\rfmod R$.
\end{enumerate}
\end{theorem}

\begin{proof} The proof of the correspondence between parts (1) - (4) is dual to the one presented in Theorem \ref{ctt}, employing Lemma \ref{1cotilt} in place of Lemma \ref{1tilt}.

(5) $\rightarrow$ (2). Let $\mathcal S$ be a $1$-resolving subcategory of $\rfmod R$. Then $\mathcal C = \mathcal S ^\intercal$ is a $1$-cotilting class in $\lmod R$ by \cite[Theorem 15.19]{GT}. The assignment $\gamma : \mathcal S \mapsto \mathcal S ^\intercal$ is monic, since $^\intercal (\mathcal S ^\intercal) = \varinjlim \mathcal S$ by \cite[Theorem 8.40]{GT}, and $\mathcal S$ coincides with the class of all finitely presented modules in $\varinjlim \mathcal S$ by \cite[Lemma 2.13]{GT}.

(2) $\rightarrow$ (5). Assume $R$ is left noetherian. Then each $1$-cotilting class in $\lmod R$ is \emph{of cofinite type} by \cite[Theorem 15.31]{GT}. That is, $\mathcal C$ of the form $\mathcal C = \mathcal S ^\intercal$ for a $1$-resolving subcategory $\mathcal S$ of $\rfmod R$, see \cite[Theorem 15.19 and Lemma 13.48]{GT}. Thus, $\gamma$ is surjective.
\end{proof}

We note some immediate corollaries of Theorems \ref{ctt} and \ref{tct}.

\begin{corollary}\label{weakclosed} Let $R$ be a ring.
\begin{enumerate}
\item[(1)] A weak tilting subcategory $\mathbb T$ of $\rmod R$ is tilting, if and only if $\mathbb T$ is closed under direct sums.
\item[(2)] A weak cotilting subcategory $\mathbb C$ of $\lmod R$ is cotilting, if and only if $\mathbb C$ is closed under direct products.
\end{enumerate}
\end{corollary}

\begin{corollary}\label{explicit} Let $R$ be a ring. Then there is a monic map assigning to each cotorsion torsion triple $(\mathcal C, \mathcal T, \mathcal F)$ in $\rmod R$ the torsion cotorsion triple $(\mathcal A, \mathcal B, \mathcal D)$ in $\lmod R$, such that $\mathcal B = \mathcal S ^\intercal $ where $\mathcal S$ is the class of all finitely presented modules in $\mathcal C$.

This map is bijective in case $R$ is left noetherian.
\end{corollary}

\begin{proof} By Theorem \ref{ctt}, cotorsion torsion triples in $\rmod R$ are parametrized by $1$-resolving subcategories $\mathcal S$ of $\rfmod R$. Given such cotorsion torsion triple $(\mathcal C, \mathcal T, \mathcal F)$, the class $\mathcal S$ is obtained as the class of all finitely presented modules in $\mathcal C = {}^\perp \mathcal T$.

The torsion cotorsion triple $(\mathcal A, \mathcal B, \mathcal D)$ in $\lmod R$ corresponding to $\mathcal S$ is defined by $\mathcal B = \mathcal S ^\intercal$, $\mathcal A = \mbox{Ker} \Hom{R}{-}{\mathcal B}$, and $\mathcal D = \mathcal B ^\perp$ (see \cite[Theorem 15.19]{GT}).

By Theorem \ref{tct}, the correspondence above is bijective when $R$ is left noetherian.
\end{proof}

\begin{example}\label{relation} If $R$ is not left noetherian, then the correspondence from (5) to (2) in Theorem \ref{tct} need not be bijective. That is, there may exist $1$-cotilting classes of left $R$-modules that are not of cofinite type. For example, if $R$ is a valuation domain, then the latter phenomenon occurs, if and only if $R$ \emph{not strongly discrete} (i.e., $R$ contains a non-zero idempotent prime ideal), see \cite{Ba}.
\end{example}

The known classification results from infinite dimensional tilting theory make it possible to give further characterizations in various particular settings:

\begin{corollary}\label{artin+noe}
\begin{enumerate}
\item[(1)] Let $R$ be an artin algebra. Then there is a bijective correspondence between the cotorsion torsion triples in $\rmod R$ and the torsion classes in $\rfmod R$ containing an injective cogenerator. It is given by the assignment $(\mathcal C, \mathcal T, \mathcal F) \mapsto  \mathcal T \cap \rfmod R$.

\item[(2)] Let $R$ be a left noetherian ring. Then there is a bijective correspondence between the torsion cotorsion triples in $\lmod R$ and the torsion-free classes in $\lfmod R$ containing $R$, given by the assignment $(\mathcal T, \mathcal F, \mathcal D) \mapsto \mathcal F \cap \lfmod R$.

\item[(3)] Let $R$ be a left noetherian ring. Denote by $\mathcal G$ the torsion class in $\lfmod R$ consisting of all $G \in \lfmod R$ such that  $\Hom{R}{G}{R} = 0$ for each $G \in \mathcal G$. Then there is a bijective correspondence between the cotorsion torsion triples in $\rmod R$ and the torsion classes $\mathcal H$ in $\lfmod R$ contained in $\mathcal G$. It is given by the assignment $(\mathcal C, \mathcal T, \mathcal F) \mapsto  \mathcal H = \{ N \in \rfmod R \mid M \otimes_R N = 0 \hbox{ for all } M \in \mathcal T \} \cap \mathcal G$.
\end{enumerate}
\end{corollary}

\begin{proof}
(1) This follows by \cite[Theorem 14.9]{GT}. The latter result shows that for artin algebras, $1$-tilting classes $\mathcal T$ in $\rmod R$ correspond bijectively to torsion classes $\mathcal T ^\prime$ in $\rfmod R$ containing all finitely generated injective modules. The correspondence is given by the mutually inverse maps $\mathcal T \mapsto \mathcal T \cap \rfmod R$ and $\mathcal T ^\prime \mapsto \mbox{Ker} \Hom{R}{-}{\mathcal F ^\prime}$ where $(\mathcal T^\prime, \mathcal F^\prime)$ is a torsion pair in $\rfmod R$.

(2) By \cite[Theorem 15.26]{GT}, for each left noetherian ring $R$, $1$-cotilting classes $\mathcal F$ in $\lmod R$ correspond bijectively to torsion-free classes $\mathcal F ^\prime$ in $\lfmod R$ containing $R$. The correspondence is given by the mutually inverse maps $\mathcal F \mapsto \mathcal F \cap \lfmod R$ and $\mathcal F ^\prime \mapsto \varinjlim \mathcal F ^\prime$.

(3) Let $(\mathcal C, \mathcal T, \mathcal F)$ be a cotorsion torsion triple in $\rmod R$. Then $\mathcal H = \{ N \in \rfmod R \mid M \otimes_R N = 0 \hbox{ for all } M \in \mathcal T \} \cap \mathcal G$ is a torsion class in $\rfmod R$ contained in $\mathcal G$. Conversely, given $\mathcal H$ as in the statement, $\mathcal T = \{ M \in \rmod R \mid M \otimes_R H = 0 \hbox{ for all } H \in \mathcal H \}$ is a $1$-tilting torsion class in $\rmod R$ by \cite[Theorem 15.31]{GT}.
\end{proof}

In the particular setting of commutative noetherian rings, Corollary \ref{explicit} gives an explicit bijective correspondence between the cotorsion torsion triples and the torsion cotorsion triples in $\rmod R$. Corollary \ref{comnoe} below makes this bijection explicit in a different way, using subsets $P$ of $\Spec R$ containing $\Ass{R}{R}$ and \emph{closed under generalization} (i.e., such that $q \in P$ whenever $p \in P$ and $q \subseteq p$).

\begin{corollary}\label{comnoe} Let $R$ be a commutative noetherian ring. Then there is a bijective correspondence between
\begin{enumerate}
\item[(1)] subsets $P$ of $\Spec R$ that are closed under generalization and contain $\Ass{R}{R}$,
\item[(2)] cotorsion torsion triples $(\mathcal C, \mathcal T, \mathcal F)$ in $\rmod R$, and
\item[(3)] torsion cotorsion triples $(\mathcal A, \mathcal B, \mathcal D)$ in $\rmod R$.
\end{enumerate}
The correspondence assigns to each subset $P$ as in (1) the cotorsion torsion triple $(\mathcal C, \mathcal T, \mathcal F)$ with $\mathcal T = \{ M \in \rmod R \mid Mq = M \hbox{ for all } q \in \Spec R \setminus P \}$, and the torsion cotorsion triple $(\mathcal A, \mathcal B, \mathcal D)$ with $\mathcal B = \{ M \in \rmod R \mid \Ass{R}{M} \subseteq P \}$.
\end{corollary}

\begin{proof} The first assignment works by \cite[Theorem 16.4]{GT}. The second follows by part (2) and the fact (\cite[Lemma 16.5]{GT}) that for a commutative noetherian ring, the torsion-free classes $\mathcal F ^\prime$ in $\rfmod R$ containing $R$ correspond bijectively to generalization closed subsets $P$ of $\Spec R$ via the assignment $\mathcal F ^\prime \mapsto \Ass{R}{\mathcal F ^\prime}$.
\end{proof}

As an application of Theorems \ref{ctt} and \ref{tct}, we will show that it is not possible to construct non-trivial cotorsion torsion triples (resp. torsion cotorsion triples) from cotorsion pairs induced by the class of Gorenstein projective or flat modules (resp. the class of Gorenstein injective modules).

\begin{corollary}\label{cor:GP-ctt} The following results are true for any ring $R$:
\begin{enumerate}
\item[(1)] $(\mathcal{GP}, \mathcal{GP}^{\perp_{1}}, (\mathcal{GP}^{\perp_{1}})^{\perp_{0}})$ is a cotorsion torsion triple in $\rmod R$ if and only if $\mathcal{GP} = \mathcal P _0$.
\item[(2)] $(^{\perp_{0}}(^{\perp_{1}}\mathcal{GI}), {^{\perp_{1}}\mathcal{GI}}, \mathcal{GI})$ is a torsion cotorsion triple in $\rmod R$ if and only if $\mathcal{GI} = \mathcal I _0$.
\item[(3)] $(\mathcal{GF}, \mathcal{GF}^{\perp_{1}}, (\mathcal{GF}^{\perp_{1}})^{\perp_{0}})$ is a cotorsion torsion triple in $\rmod R$ if and only if $\mathcal{GF} = \mathcal P _0$.
\end{enumerate}
\end{corollary}
\begin{proof} (1)  It suffices to show the ``only if'' part. By Theorem \ref{ctt}, $(\mathcal{GP}, \mathcal{GP}^{\perp_{1}})$ is a 1-tilting cotorsion pair, and hence the projective dimension of any Gorenstein projective right module is at most 1 by Lemma \ref{1tilt}. This forces $\mathcal{GP}$ to be exactly the class of projective right $R$-modules.

(2) The argument is analogous to that of (1).

(3) Again it suffices to show the ``only if'' part. It follows from Theorem \ref{ctt} that $(\mathcal{GF},\mathcal{GF}^{\perp_{1}})$ is a 1-tilting cotorsion pair, so in particular $\mathcal{GF}^{\perp_{1}}$ is closed under direct sums. Observe that $\mathcal{GF}\cap \mathcal{GF}^{\perp_{1}}$ is precisely the class of flat and cotorsion right $R$-modules. Therefore, by a result of Herzog and Guil Asensio \cite[Theorem 19]{GH}, the ring $R$ is right perfect. It follows that $\mathcal{GF}$ is exactly the class of projectively coresolved Gorenstein flat right $R$-modules, whence $\mathcal{GF}\subseteq\mathcal{GP}$. Because $(\mathcal{GF},\mathcal{GF}^{\perp_{1}})$ is a 1-tilting cotorsion pair, the projective dimension of any Gorenstein flat right module is at most 1. Thus $\mathcal{GF}$ is exactly the class of projective right $R$-modules, as desired.
\end{proof}

Let $R$ be a ring. Recall that the \emph{little finitistic dimension}
of $R$, $\textrm{findim}(R)$, is the supremum of projective dimensions of all
finitely generated modules of finite projective dimension. It is
conjectured that $\textrm{findim}(R)<\infty$ holds for all artin algebras $R$; see
Bass \cite{HBs60} and \cite[Conjectures]{rta}. This is the
Finitistic Dimension Conjecture. An application of Theorem \ref{ctt} yields the following result.

\begin{corollary}\label{cor:fin.dim} Let $R$ be a right noetherian ring, and let $\mathcal{P}^{<\infty}(R)$ denote the class of finitely generated modules of finite projective dimension. Set $\mathcal{V} = (\mathcal{P}^{<\infty}(R))^{\perp_{1}}$. Then the following conditions are equivalent:
\begin{enumerate}
\item[(1)] ${\rm findim}(R)\leq1$.
\item[(2)] $({^{\perp_{1}}{\mathcal{V}}},\mathcal{V}, \mathcal{V}^{\perp_{0}})$ is a cotorsion torsion triple in $\rmod R$.
\end{enumerate}
\end{corollary}
\begin{proof} $(1)\Rightarrow(2)$. Since ${\rm findim}(R)\leq1$, there is a 1-tilting module $T$ such that
$\mathcal{V}=T^{\perp_{\infty}}$ is a 1-tilting torsion class in $\rmod R$ by \cite[Theorem 2.6]{AT02} and \cite[p.337]{GT}. Therefore, $({^{\perp_{1}}{\mathcal{V}}},\mathcal{V}, \mathcal{V}^{\perp_{0}})$ is a cotorsion torsion triple in $\rmod R$ by Theorem \ref{ctt}.

$(2)\Rightarrow(1)$. Theorem \ref{ctt} implies that $\mathcal{V}$ is a 1-tilting class. Hence ${\rm findim}(R)\leq1$ by \cite[Theorem 2.6]{AT02}.
\end{proof}

\end{document}